\documentclass[12pt,a4paper,twoside,reqno]{amsart}

\usepackage{tikz}
\usepackage{amsmath}
\usepackage{amssymb,amsfonts,mathrsfs}
\usepackage[colorlinks=true,linkcolor=blue,citecolor=blue]{hyperref}
\usepackage[driver=pdftex,margin=3cm,heightrounded=true,centering]{geometry}

\usepackage{mllocal} 

\numberwithin{equation}{section}

\begin{document}

\title[Regular singular operators,  zeta-determinants]
{Regular singular Sturm-Liouville operators and their zeta-determinants}

\author{Matthias Lesch}
\address{Mathematisches Institut,
Universit\"at Bonn,
Endenicher Allee 60,
53115 Bonn,
Germany}

\email{ml@matthiaslesch.de, lesch@math.uni-bonn.de}
\urladdr{www.matthiaslesch.de, www.math.uni-bonn.de/people/lesch}

\author{Boris Vertman}
\address{Mathematisches Institut,
Universit\"at Bonn,
Endenicher Allee 60,
53115 Bonn,
Germany}
\email{vertman@math.uni-bonn.de}
\urladdr{www.math.uni-bonn.de/people/vertman}

\thanks{Both authors were partially supported by the 
        Hausdorff Center for Mathematics.}
\thanks{The second author thanks Stanford University for hospitality.}


\begin{abstract}
We consider Sturm-Liouville operators on the line segment $[0,1]$ with general
regular singular potentials and separated boundary conditions.  We establish
existence and a formula for the associated zeta-determinant in terms of the
Wronski-determinant of a fundamental system of solutions adapted to the
boundary conditions. This generalizes the earlier work of the first author,
treating general regular singular potentials but only the Dirichlet boundary
conditions at the singular end,  and the recent results  by Kirsten-Loya-Park
for general separated boundary conditions but only special regular singular
potentials.
\end{abstract}

\maketitle
\tableofcontents
\listoffigures

\section{Introduction and formulation of the result}

In this paper we will investigate the zeta-determinant of Sturm--Liouville
operators of the form
\[
H= -\frac{d^2}{dx^2} + \frac{\nu^2-1/4}{x^2} + \frac 1x V(x), \quad \Re \nu
\ge 0, x\in (0,1),
\]
the regularity assumptions on $V$ will be minimal.

Such operators are sometimes also called Bessel operators. $H$ is
the prototype of a differential expression with one regular singularity
and hence it appears naturally in the classical theory of ordinary
differential equations with regular singularities \cite{CodLev:TOD}. 
The physical relevance of $H$ stems from the fact that it arises
when separation of variables is used for the 
radial Schr\"odinger operator in Euclidean space.

Quite recently there has been a lot of interest in the inverse spectral theory
of $H$, see e.g. \cite{KST:WTT}, \cite{KST:IEP}, \cite{Car:BLT} and the
references therein.

Our motivation for looking at the zeta-determinant of $H$ comes from geometry:
Spectral geometry on manifolds with singularities has been initiated by
Cheeger in his seminal papers \cite{Che:SGS}, \cite{Che:SGSR}.  Manifolds with
conical singularities are an important case study for this general programme.
Separation of variables for the Laplacian on a cone leads to an infinite sum
of Bessel type operators like $H$ above.  Recently, there has been a revived
interest in extending the celebrated Cheeger-M\"uller Theorem \cite{Che:ATH},
\cite{Mul:ATT} on the equality of the analytic torsion and the Reidemeister
torsion to manifolds with conic singularities \cite{DaiHua:IRT},
\cite{deMHarSpr:RTA}, \cite{HarSpr:ATC}, \cite{Ver:ZDR}, \cite{Ver:MAR}.

The separation of variables mentioned above leads naturally to the problem
of determining the zeta-determinant of a single regular singular Sturm-Liouville operator
on the line segment $[0,1]$  with separated boundary conditions.
We only make minimal regularity assumptions on the potential. 
Nevertheless, we establish existence and a formula for the
associated zeta-determinant in terms of the  Wronskian of a
fundamental system of solutions adapted to the boundary conditions, see
Theorem \plref{thm:Main} below.

We emphasize that for the calculation of the analytic torsion or the
zeta-determinant on a cone additional considerations are necessary. This
is because on a cone one has to deal with an \emph{infinite} direct sum of operators
like $H$.

The fundamental results of Br\"uning-Seeley in \cite{BruSee:RSA}, \cite{BruSee:RES} guarantee the 
existence of zeta-determinants for regular singular Sturm-Liouville operators with 
Dirichlet boundary conditions at the singularity. However,
loc. cit. require the potential to be of the form $a(x)/x^2$ with $a(x)$ smooth up to $0$. 
For such operators Theorem \plref{thm:Main} was proved in \cite{Les:DRS} by the first author,
generalizing earlier results by Burghelea, Friedlander and Kappeler \cite{BFK:DEB} to the regular singular setting.

The method of \cite{Les:DRS} is limited to the Friedrichs extension at the singularity. 
In a recent series of papers Kirsten, Loya, and Park \cite{KLP:FDG}, \cite{KLP:EEP}, \cite{KLP:UPR} 
were able to calculate the zeta-determinant for an explicit example of a regular singular 
Sturm-Liouville operator with general self-adjoint boundary conditions; 
cf. also the subsequent discussion by the second author in \cite{Ver:ZDR} and
in the appendix to \cite{KLP:EEP}.
Their method, however, is based on an intricate analysis of Bessel functions and is therefore limited to their explicit potential.

The main result of this paper combines and generalizes these two results, however only for scalar valued potentials. 
Since we deal with a rather general class of potentials it is natural
that our method is closer to that of \cite{Les:DRS}. Special functions are used in this paper only implicitly as we
are using the formula from \cite{Les:DRS} for the zeta-determinant of the Friedrichs extension of the regular singular model 
operator $l_\nu=-\frac{d^2}{dx^2}+(\nu^2-1/4)/x^2$.

The paper is organized as follows.  In the remainder of this section  we will
introduce some notation, explain the basic concepts of regularized integrals
and zeta-determinants, and we will formulate our main result.  In Section
\ref{sec:FunSys} we derive the asymptotic behavior of a fundamental system for
$H$, slightly generalizing a result due to B\^ocher \cite{Boc:RSP}.

In Section \ref{s-a} we study the maximal domain of $H$ and its closed
extensions with separated boundary conditions. Let $H(\theta_0,\theta_1)$
be such an extension, $\theta_0,\theta_1$ stand for the boundary conditions at
$0,1$ resp. We give criteria under which it is possible to factorize
$H(\theta_0,\theta_1)$ into a product $D_1D_2$ of closed extensions of
first order regular singular differential operators.
We prove a comparison result for the Wronskians
of normalized fundamental solutions for $D_1D_2$ and $D_2D_1$. 

In Section \ref{sec:ResolventExpansion} we discuss the asymptotic expansion of
the resolvent trace.  We start with the Friedrichs extension. The resolvent of
the Friedrichs extension $L_\nu$ of the regular singular model operator
$l_\nu$ is explicitly known and rather well--behaved with respect to
perturbations of the form $X\ii V$. From \cite{BruSee:RSA} we only use the
result that the resolvent of the model operator $L_\nu$ has a complete
asymptotic expansion. The expansion of the resolvent of $L_\nu+X\ii V$
then follows by a perturbation analysis. Boundary conditions other than the
Friedrichs extension at $0$ are more subtle since the resolvent does not
absorb high enough negative powers of $x$. For the resolvent of general
boundary conditions we therefore use the factorization results of Section
\plref{s-a}.  In addition one needs to
treat compactly supported $L^2$-perturbations of factorizable operators.
For this we employ a standard method of pasting together local resolvents, 
cf. \cite[Appendix]{LMP:CCC}.

In Section \ref{q-section} we derive a variational formula for the dependence of the zeta-determinant
under variation of the potential. The method is well--known \cite{BFK:DEB}, \cite{Les:DRS}. However,
due to the low regularity assumptions on the potential and due to the singularity of the operator the
analysis becomes a little delicate. In particular we have to analyze the dependence of a normalized fundamental
system (and its asymptotic behavior near $0$) on the parameter.
At the end of Section \plref{q-section} we compile the established results to
a proof of the main Theorem \ref{thm:Main}.

\section*{Acknowledgements}

We are indebted to the referees who went far beyond the call of duty and
provided a comprehensive list of suggestions for improvements. We think
the paper has benefited a lot from these suggestions and we would like
to express our gratitude.

\subsection{Function and Distribution Spaces}\label{ss:notation}
Following the requests of several of the referees we are going to specify in detail
the notation for function and distribution spaces used throughout the paper. 

Let $I\subset \R$ be an interval, 
which may be of any of the possible forms $(a,b), [a,b), (a,b]$ or $[a,b]$ for
real numbers $a<b$. Let $I^\circ=I\setminus \{a,b\}$
denote the interior of $I$. 

For a map $f:I\to E$ into some vector space $E$ the support of $f$, denoted by $\supp f$, is defined as the closure in $I$ of $\{x\in I\mid f(x)\neq 0\}$
\begin{align}\label{support}
 \supp f:=\overline{\{x\in I\mid f(x)\neq 0\}}^{\, I}.
\end{align}
$\supp f$ is always closed in $I$ but not necessarily compact, since $I$ might be non-compact itself. 

For spaces of continuous, respectively differentiable complex-valued functions
we use the standard notation $C(I), C^k(I), C^{\infty}(I)$, cf. e.g. \cite[Sec. 1.1, 1.2]{Hoe:TAL}. 
The space $C^k_0(I), 0\leq k\leq \infty,$ denotes the subspace of those $f\in C^k(I)$ with 
compact support. 

The space $C^{\infty}_0(I^\circ)$ carries a natural locally convex topology and its dual space $\dom'(I^\circ)$ is called the space of distributions on $I^\circ$. 

For $T\in \dom'(I^\circ)$ one can define $\supp T$ \cite[Sec. 2.2]{Hoe:TAL}. 
For an \emph{arbitrary} subset $A\subset \R$ one now writes 
$\mathscr{E}'(A)=\{T\in \dom'(\R)\mid \supp T\subset A\}$, cf. \cite[Sec. 2.3]{Hoe:TAL}. For the half open interval $I=(a,b]$, e.g., $T\in \mathscr{E}'((a,b])$ if there is a 
$\delta >0$ such that $\supp T\subset (a+\delta, b]$.

For distributions it also makes sense to talk about restrictions. If $J\subset I$ are intervals and $T\in \dom'(I^\circ)$, we put $T|_J:=T\restriction C^{\infty}_0(J^\circ)$.

Let $F$ be a map which assigns to each interval $I\subset \R$ a subspace $F(I)\subset \dom'(I^\circ)$. Furthermore, assume that $F$ is compatible with restrictions in the following sense:
if $J\subset I$ are intervals and $f\in F(I)$, then $f|_J \in F(J)$. Then $F_\loc(I)$ denotes the space of $T\in \dom'(I^\circ)$ such that $T|_K\in F(K)$ for each compact interval
$K\subset I$. Furthermore, $F_\comp(I):=F_\loc(I)\cap \mathscr{E}'(I)$. 

\begin{example}
For an interval $I\subset \R$ we denote by $L^p(I), 1\leq p \leq \infty,$ the
Banach space of $p$-summable (equivalence classes modulo equality almost everywhere) 
functions with respect to Lebesque measure; for $f\in L^p(I)$ the norm is given by 
$$\|f\|_{L^p}:=\Bigl(\int_I |f(x)|^p dx\Bigr)^{1/p}.$$ 
$L^p(I)$ is naturally embedded into $\dom'(I^\circ)$ by identifying $f\in L^p(I)$ 
with the distribution 
\begin{align}\label{distribution}
 C^{\infty}_0(I^\circ)\owns \phi \mapsto \int_I f\cdot \phi;
\end{align}
needless to say that \eqref{distribution} is independent of the choice of a function 
representative of the class $f$. The support of $f\in L^p(I)$ is now defined as the 
\emph{closure in $\overline{I}$} of the support of the corresponding distribution in 
$\dom'(I^\circ)$. For continuous functions $C(I)\subset L^p(I)$ (each $L^p$-class has 
at most one continuous representative) the latter definition of support coincides with
\eqref{support}, assuming that $I$ is closed.

The assignment $I\mapsto L^p(I)$ is an example for the map $F$ discussed
above. Hence, $L^p_\comp(I)$ and $L^p_\loc(I)$ are
defined. Note that although $L^p(I)=L^p(\overline{I})$, we only have 
$$L^p_\comp(I) \subset L^p_\comp(\overline{I}), \quad
L^p_\loc(\overline{I}) \subset L^p_\loc({I}).$$
\end{example}

Sobolev spaces will only be used in the Hilbert space setting $p=2$. We write $H^k(I)$ for the Sobolev space $W^{k,2}(I)$ of those $f\in L^2(I)\subset \dom'(I^\circ)$, for which all weak distributional derivatives $\partial^jf, 1\leq j \leq k,$ taken a priori in $\dom'(I^\circ)$, are actually in $L^2(I)$.

\subsubsection{The Schatten Ideals} 
For a Hilbert space $\mathcal{H}$ we denote by $\mathcal{B}(\mathcal{H})$ the space of bounded 
and by $\mathcal{K}(\mathcal{H})$ the space of compact operators on $\mathcal{H}$. 
For $1\leq p<\infty$ let $\mathcal{B}^p(\mathcal{H})\subset \mathcal{K}(\mathcal{H})$ be the 
von Neumann-Schatten ideal of $p$-summable operators, cf. e.g. \cite[Sec. 3.4]{Ped:AN}. 
For $T\in \mathcal{B}^p(\mathcal{H})$ the $p$-norm is given by
\[
\|T\|_p:=\bigl(\Tr(T^*T)^{p/2}\bigr)^{1/p}=
\Bigl(\sum_{\lambda\in \spec T^*T}\lambda^{p/2}\Bigr)^{1/p}.
\]
Tr denotes the trace \cite[Sec. 3.4]{Ped:AN}. We will only need $p=1$ and $p=2$. Operators in 
$\mathcal{B}^1(\mathcal{H})$ are called trace class operators and elements of 
$\mathcal{B}^2(\mathcal{H})$ are called Hibert-Schmidt operators. To avoid possible confusion 
with the $L^p$-norm of functions, we write $\|\cdot\|_{\tr}$ for the trace norm 
$\|\cdot \|_1$ and $\|\cdot\|_{\HS}$ for the Hilbert-Schmidt norm $\|\cdot \|_2$.  

\subsubsection{Regularized Integrals}
Let us briefly recall the partie finie regularization, cf. \cite[Sec. 2.1]{Les:OFT}, 
\cite{Les:DRS} and \cite{LesTol:DOD}, of integrals on $\R_+:=[0,\infty)$. 
Let $f:(0,\infty)\to\C$ be a locally integrable function. Assume furthermore, that 
for $x\geq x_1$ we have a representation
\begin{align}\label{infty-asympt}
f(x)=\sum_{j=1}^N f^{\infty}_j x^{\A_j}+g(x),
\end{align}
with real numbers $\A_j$, numbered in descending order with $\A_N=-1$, and $g\in L^1[x_1,\infty)$. 
Then
\begin{equation}\label{regint-infty}
\begin{split}
 \int_{x_1}^R &f(x) dx\\
    &=:\sum_{j=1}^{N-1}\frac{f^{\infty}_j}{\A_j+1}R^{\A_j+1}+f^{\infty}_N\log R + 
       \regint_{x_1}^{\infty}f(x)dx +o(1),  \textup{ as }  R\to \infty.
\end{split}
\end{equation}
$o(1), R\to \infty,$ is the usual Landau notation for a function of $R$ whose
limit as $R\to \infty$ is zero; here we have explicitly
$o(1)=\int_{R}^{\infty}g$. The regularized integral
$\regint_{x_1}^{\infty}f(x)dx$ is therefore defined as the constant term in
the asymptotic expansion of $\int_{x_1}^Rf(x)dx$ as $R\to \infty$. 

If for $0<x\leq x_0$ we have a representation
\begin{align}\label{zero-asympt}
 f(x)=\sum_{j=1}^M f^{0}_j x^{\beta_j}+h(x),
\end{align}
with real numbers $\beta_1<\beta_2<\cdots <\beta_M=-1$, and $h\in L^1[0,x_0]$, then
\begin{align}\label{regint-zero}
 \begin{split}
 \int_{\delta}^{x_0} &f(x)dx  \\
     &=:-\sum_{j=1}^{M-1}\frac{f^{0}_j}{\beta_j+1}\delta^{\beta_j+1}-
      f^{0}_M\log \delta + \regint_{0}^{x_0}f(x)dx +o(1), \textup{ as } \delta\to 0,
     \end{split}     
\end{align}
and the regularized integral $\regint_{0}^{x_0}f(x)dx$ is defined as the
constant term in the asymptotic expansion of $\int_{\delta}^{x_0}f(x)dx$ as
$\delta\to 0$. 

Now assume that $f$ satisfies \eqref{infty-asympt} and \eqref{zero-asympt}.
Since $f$ is locally integrable, it is clear that \eqref{regint-infty} holds
for any $x_1>0$ and \eqref{regint-zero} holds for any $x_0>0$. One then puts
for any $c>0$
\begin{align}\label{regint}
\regint_{0}^{\infty}f(x)dx:=\regint_{0}^{c}f(x)dx+\regint_{c}^{\infty}f(x)dx,
\end{align}
and in fact the right hand side is independent of $c>0$.

\subsection{The zeta-determinant}\label{ss:ZetaDeterminant}
Let $H$ be a closed not necessarily self-adjoint operator acting on some Hilbert space with $\spec(-H)\cap \R_+$ finite, $0\not\in\spec H$.
We assume that the resolvent of $H$ is trace class,  and that for $z\in\R, z\ge z_0>\max \bigl(\spec(-H)\cap \R_+\bigr)$
\begin{equation}\label{eq:DetExp}
   \Tr( H+z)\ii= \frac{a}{\sqrt{z}}+\frac bz+R(z)
\end{equation} 
with
\begin{align}
&\lim_{z\to\infty} z R(z)=0,\label{eq:DetConA}\\
&\int_{z_0}^\infty |R(z)| dz<\infty.\label{eq:DetConB}
\end{align}
\begin{figure}[t]
\begin{center}
\begin{tikzpicture}
\draw (-2.5,0) -- (5,0);
\draw[->] (0,-2) -- (0,3);

\node at (1.5,0) {$\bullet$};
\node at (2.5,0) {$\bullet$};
\node at (3.5,0) {$\bullet$};
\node at (-1,0) {$\bullet$};
\node at (-2,0) {$\bullet$};
\node at (-3,0) {$\cdots$};

\node at (-1,1) {$\bullet$};
\node at (-1,2) {$\bullet$};
\node at (-2,1) {$\cdots$};
\node at (-2,2) {$\cdots$};
\node at (2,1) {$\bullet$};
\node at (1,1) {$\bullet$};
\node at (1,2) {$\bullet$};

\node at (0,1) {$\bullet$};

\node at (-1,-1) {$\bullet$};
\node at (-2,-1) {$\cdots$};

\draw[very thick] (5.8,0.1) node[below=10pt] {\large{$\Gamma$}} -- (6,0);
\draw[very thick] (5.8,-0.1) -- (6,0);


\draw[very thick] (0,0) -- (0.5,0);
\draw[very thick] (0.5,0) .. controls (1,0) and (1,0.5) .. (1.5,0.5);
\draw[very thick] (1.5,0.5) -- (3.5,0.5);
\draw[very thick] (3.5,0.5) .. controls (4,0.5) and (4,0) .. (4.5,0);
\draw[very thick] (4.5,0) -- (7,0);

\end{tikzpicture}
\end{center}
\caption{The contour of integration $\Gamma$.}
\label{fig:Contour}
\end{figure}
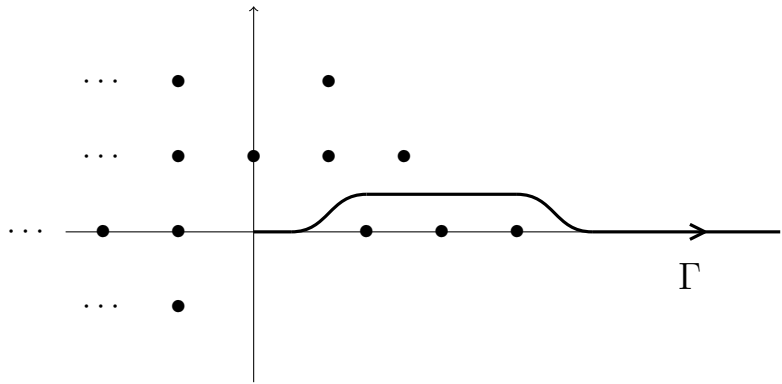

Let $\Gamma$ be the contour as sketched in Figure \plref{fig:Contour}, 
where the bullets indicate the eigenvalues of $-H$.
Fix a branch of the logarithm in the simply connected domain $\C\setminus \{-\Gamma\}$. 
Note that the previous definition \eqref{regint} of the regularized integral can easily be adapted 
to functions defined on the contour $\Gamma$, since there are $0<x_0<x_1<\infty$ such that $[0,x_0]$ 
and $[x_1,\infty)$ are contained in (the image of) $\Gamma$. Consider for fixed $s\in \C$ the function 
$$f_s(x):=x^{-s}\, \Tr( H+x )\ii.$$
In view of \eqref{eq:DetExp} and \eqref{eq:DetConB} it satisfies \eqref{infty-asympt} if $\Re s\geq 0$. 
Furthermore, since $H$ is assumed to be invertible, the function $x\mapsto \Tr(H+x)^{-1}$ is smooth 
up to $x=0$ and its Taylor expansion at $x=0$ shows that $f_s$ satisfies \eqref{zero-asympt} for all $s\in \C$. 

Exploiting the definition of $\regint_{\Gamma}$ it is now not hard to see,
cf. \cite[(2.30)]{LesTol:DOD}, that for $1<\Re s<2$ the zeta-function is given by 
\begin{equation}\label{eq:ZetaFunction}
     \zeta_H(s):=\sum_{\gl\in\spec H\setminus\{0\}} \gl^{-s}= \frac{\sin \pi s}{\pi} 
     \regint_\Gamma x^{-s} \, \Tr(H+x)\ii dx.
\end{equation}

Furthermore using the asymptotic expansions as $x\to \infty$ and $x\to 0$ of 
$x^{-s}\, \Tr(H+x)^{-1}$ one deduces that the right hand side of \eqref{eq:ZetaFunction} 
extends meromorphically to the half plane $\Re s > 0$, \cite[Prop. 2.1.2]{LesTol:DOD}. 
The identity \eqref{eq:ZetaFunction} persists except for the poles of the function 
$s\mapsto \frac{\pi}{\sin \pi s} \zeta_H(s)$. Thanks to \eqref{eq:DetConB} the function $\zeta_H$ is
differentiable from the right at $s=0$ and one puts
\begin{equation}\label{eq:ZetaPrimeZero}
\log\detz H:=-\zeta_H'(0)=- \regint_\Gamma \Tr(H+x)\ii dx.
\end{equation}

$\detz H$ is called the \emph{zeta-regularized determinant} of $H$. For non-invertible $H$
one puts $\detz H=0$. With this setting the function $z\mapsto \detz (H+z)$ is an entire holomorphic
function with zeroes exactly at the eigenvalues of $-H$. The multiplicity of a zero $z$ equals the
algebraic multiplicity of the eigenvalue $z$.

\subsection{A regular singular operator}
We now introduce the class of operators we are going to
study in this paper. Put
\begin{equation}\label{eq:ModelOperator}
l_\nu= -\frac{d^2}{dx^2} + \frac{\nu^2-1/4}{x^2}, \quad \Re \nu \ge 0,
\end{equation}
acting as an operator in the Hilbert space $L^2[0,1],$ a priori with domain $C^{\infty}_0(0,1)$. 
We will study perturbations of $l_\nu$ of the form
\begin{equation}\label{eq:PerturbationOperator}
H  = l_\nu+X\ii V,
\end{equation}
with suitable conditions on the operator $V$ to be specified below. $X$ denotes the function $X(x)=x$. 
We view $V$ resp. $X^{-1}V$
as a multiplication operator on functions on the unit interval. In order not to overburden the notation
we will in general not distinguish between a function $f$ and the operator of multiplication by $f$.

\begin{defn}\label{def:FunctionSpaces}
\textup{1. }
For an interval $I\subset \R$ we denote by $AC^k(I), k\geq 1, $ the space of 
$f\in C(I)\subset \dom'(I^\circ)$ such that  
$\partial^jf\in C(I), 0\leq j \leq k-1,$ and $\partial^kf\in L^1_\loc(I)$. $AC^1(I)=AC(I)$ 
is the well--known space of absolutely continuous functions. 
Note that for this definition it matters whether a boundary point $p\in \partial I$ belongs to $I$ or not.
 
\textup{2.} Denote by $\sV_\nu$ the space of those $V\in L^2_\loc(0,1)$ such that
\begin{align}
& V \cdot\llog\in L^1[0,1],            \quad \textup{if } \nu \neq 0, \label{eq:ML1003052}\\
& V \cdot\llogv{2}\in L^1[0,1],        \quad \textup{if } \ \nu=0.  \label{eq:ML1003053}
\end{align}
A natural norm on $\sV_\nu$ is given by 
\begin{align}
\|V\|_{\sV_\nu} &= \bigl\| \llog V\bigr\|_{L^1}, \quad \textup{if } \nu \neq 0, \\
\|V\|_{\sV_\nu} &= \bigl\| \llog^2 V\bigr\|_{L^1}, \quad \textup{if } \nu=0.
\end{align}
$\sV_\nu$ is a Fr\'echet space with seminorms $\|\cdot\|_{\sV_\nu}$ and $\|V_{\bigl|[1/n,1-1/n]}\|_{L^2}, n=2,3,\ldots$.

\textup{3.} Finally, let $\sA$ be the space of those $f\in AC^2(0,1)$ such that
$f',X f''\in \llog\ii L^1[0,1]$.
\end{defn}
Some of the results will hold under the weaker hypothesis $V\in \lsV_\nu\supset \sV_\nu$.
Unless said otherwise, function spaces consist of complex valued functions.

In Section \plref{sec:FunSys} we will prove the following refinement
of a Theorem of B\^{o}cher \cite{Boc:RSP} (Theorem \ref{5-a-t} and Proposition \plref{les-4}):

\begin{theorem} Let $V\in \sV_{\Re \nu}, H=l_\nu+X\ii V$, $\Re \nu\ge 0$, and let 
$\nu_1=\nu+1/2, \nu_2=-\nu+1/2$ be the characteristic exponents of the regular singular
point $0$ of the differential equation $H g=0$. Then there is a fundamental system $g_1,g_2$ of solutions 
to the equation $H g=0$ such that
\begin{align}
 g_1(x) =&  x^{\nu_1}\widetilde{g}_1(x),\label{eq:g1}\\
 g_2(x) =& \begin{cases} 
             -\frac{1}{2\nu}x^{\nu_2}\widetilde{g}_2(x),  &\textup{if } \nu \neq 0, \\
             \sqrt{x}\log (x) \widetilde{g}_2(x),        &\textup{if } \nu=0,
            \end{cases}
            \label{eq:g2}
\end{align}
where $\tilde g_j\in\sA$. 
\end{theorem}

The spectra and fundamental system of solutions to Bessel type
Sturm-Liouville differential expressions on finite intervals have also been
studied 
(mainly in connection with the inverse spectral problem) in a number of recent publications \cite{AHM:ISP}, \cite{Car:BLT}, 
\cite{EveKal:BDE}, \cite{KST:IEP}, \cite{KST:WTT} and \cite{Ser:ISP}.

\subsection{Separated boundary conditions}\label{ss:SepBouCon}
Denote by $H_0$ the differential expression $H$ restricted to $\cinfz{(0,1)}\subset L^2[0,1]$.
Let $H_0^t$ be the formal adjoint of $H_0$. This is the differential expression 
$-\frac{d^2}{dx^2}+\frac{\ovl{\nu}^2-1/4}{x^2}+\frac 1x \ovl{V}(x)$ acting on $\cinfz{(0,1)}$.
$H_0$ is symmetric if both $\nu\in \R$ and $V$ is real valued.

As usual we denote by $\Hmin=\overline{H_0}$ the closure of $H_0$ and by $\Hmax=(H_0^t)^*=\bigl(H^t_{\min}\bigr)^*$. 
For convenience we introduce the \emph{left minimal and right maximal domain}
$\domL(H)$ as the domain of the closure of $H\restriction \cinfz{(0,1]}$.
The \emph{left maximal and right minimal domain} $\domR(H)$ is defined accordingly with $\cinfz{(0,1]}$ replaced
by $\bigsetdef{f\in\dom(\Hmax)}{\supp f \subset [0,1) \text{ compact}}$. 
Note that by the definition of support in \eqref{support}, compactness of $\supp f\subset [0,1)$ 
means that $\supp f$ has a positive distance from the point $x=1$.

Although there is no simple Weyl alternative in the non-self-adjoint context, we say that 
$x=0$ (resp. $x=1$) is in the limit point case for $H$ if 
$\dom_L(H)=\dom(H_{\max})$ (resp. $\dom_R(H)=\dom(H_{\max})$). Otherwise, we say that it is in the 
limit circle case.

We will see in Section \ref{sec:MaxDom} that there are continuous linear
functionals $c_j, j=1,2,$ on $\dom(\Hmax)$
such that for $f\in\dom(\Hmax)$
\begin{equation}
   f=c_1(f) g_1+ c_2(f) g_2+ \tilde f,\quad \tilde f\in \domL(H),
\end{equation}
where $g_1,g_2$ are defined in \eqref{eq:g1} and \eqref{eq:g2}.

A boundary condition at the left end point is therefore of the form
\begin{equation}\label{eq:BouOp-1}
    B_{0,\theta}f:=\sin\theta \cdot c_1(f)+\cos\theta \cdot c_2(f)=0,\quad 0\le\theta<\pi.
\end{equation}
$\theta=0$ gives the Dirichlet boundary condition (Friedrichs extension near $0$).

It should be noted here that $0$ is in the limit point case for $H$ if and
only if $\Re \nu\ge 1$. In this
case $g_2\not\in L^2[0,1]$, $c_2=0$, and hence $\domL(H)=\dom(\Hmax)$. Thus
if $\Re \nu\ge 1$ we consider only the case $\theta=0$. Boundary conditions
such as \eqref{eq:BouOp-1} at the singular end point have been 
studied in depth by Rellich \cite{Rel:DZR} and extended by Bulla-Gesztesy \cite{BulGes:DIS}. 

From the well--known fact that a linear second order ODE with $L^1$--coefficients has 
$AC^{2}$--solutions it follows in view of our assumptions on $V$ that 
$\dom(\Hmax)\subset AC^{2}(0,1]$
and hence  $1$ is always in the limit circle case for $H$. 
At the right end-point we therefore impose boundary conditions of the form
\begin{equation}\label{eq:BouOp-2}
    B_{1,\theta}f:=\sin\theta\cdot f'(1)+\cos\theta\cdot f(1)=0, \quad 0\le \theta<\pi.
\end{equation}
For each admissible pair $(\theta_0,\theta_1)\in [0,\pi)^2$ ($0\le \theta_0<\pi$ if $\Re\nu<1$, 
$\theta_0=0$ if $\Re \nu\geq 1$) we obtain a closed realization $H(\theta_0,\theta_1)$ of the 
operator with separated boundary conditions $B_{0,\theta_0}, B_{1,\theta_1}$. All eigenvalues of $H(\theta_0,\theta_1)$
are therefore simple.

Under the technical assumption that $V$ is of determinant class, see Definition \plref{def:DetClass}, which is satisfied for
all real valued potentials $V\in\sV_\nu$ we can prove (Theorem \plref{5-b-t} and Theorem \plref{t:TraceAsymptotics})
\begin{theorem}\label{thm:ResolventExpansion} 
Let $\nu\ge 0$, $V\in\sV_\nu$, and assume that $\theta_0=0$ or that $V$ is of determinant class and $\nu>0$.
Then the resolvent of $H(\theta_0,\theta_1)$ is trace class. Moreover, there is a $z_0>0$ such
that $H(\theta_0,\theta_1)+z$ is invertible for $z\ge z_0$ and there is an asymptotic expansion
\begin{equation}
   \Tr\bigl( H(\theta_0,\theta_1)+z\bigr)\ii= \frac{a}{\sqrt{z}}+\frac bz+R(z), \quad z\ge z_0, z\in \R,
\end{equation}
with
\begin{align}
    &\lim_{z\to\infty} z R(z)=0,\\
    &\int_{z_0}^\infty |R(z)|dz <\infty.
\end{align}
\end{theorem}

In view of this Theorem we may define $\detz\bigl( H(\theta_0,\theta_1)\bigr)$ according to \eqref{eq:ZetaPrimeZero}.

\subsection{The main result}
To explain our result we need to introduce the notion of a normalized solution at one of the end points.
First, we define an invariant of the boundary operator $B_{j,\theta}$:
\begin{equation}\label{eq:DefMu0}
 \mu_0:=\mu(B_{0,\theta}):=\begin{cases} \nu,& \textup{if } \theta=0,\\
                                       -\nu, & \textup{if } 0<\theta<\pi;
                         \end{cases}
\end{equation}
resp.
\begin{equation}\label{eq:DefMu1}
      \mu_1:=\mu(B_{1,\theta}):=\begin{cases} 1/2,& \textup{if } \theta=0,\\
                                       -1/2, & \textup{if } 0<\theta<\pi.
                         \end{cases}
\end{equation}
To explain the $\pm 1/2$ we note that the right end point may artificially be viewed as
a regular singular point with $\nu=1/2$. Hence $\mu_j$ depend in fact on $\theta$ and
the characteristic exponent of the regular singular point.

A solution of the homogeneous differential equation $Hg=0$ is called \emph{normalized
at $0$} with respect to the boundary operator $B_{0,\theta}$ if $B_{0,\theta}g=0$ and
if $g(x)\sim x^{\mu_0+1/2}$, as $x\to 0$; here we use the notation
\begin{align}\label{sim}
f(x)\sim h(x), \ \textup{as} \ x\to x_0 : \Leftrightarrow \ \lim_{x\to x_0}\, \frac{f(x)}{h(x)}=1.
\end{align}
 Similarly, $g$ is called \emph{normalized at $1$}
with respect to the boundary operator $B_{1,\theta}$ if $B_{1,\theta}g=0$ and
if $g(x)\sim (1-x)^{\mu_1+1/2}$ as $x\to 1$. 
It is straightforward to check that there is always a unique normalized solution.

\begin{theorem}\label{thm:Main}
Let $B_{j,\theta_j}, j=0,1$ be admissible boundary operators for $H$. Under the same assumptions as
in Theorem \plref{thm:ResolventExpansion} the zeta--regularized
determinant of $H(\theta_0,\theta_1)$ is given by
\begin{equation}\label{eq:Main}
    \detz\bigl( H(\theta_0,\theta_1)\bigr)= \frac{\pi}{2^{\mu_0+\mu_1}\Gamma(\mu_0+1)\Gamma(\mu_1+1)}W(\psi,\varphi).
\end{equation}
Here, $\varphi,\psi$ are solutions to the homogeneous differential equation $Hg=0$ such
that $\varphi$ is normalized for $B_{0,\theta_0}$ (at $0$) and $\psi$ is normalized for $B_{1,\theta_1}$
(at $1$). Furthermore, $W(\psi,\varphi)=\psi \varphi'-\psi'\varphi$ denotes the Wronskian of $\psi,\varphi$.
\end{theorem}

Theorem \plref{thm:Main} has a relatively straightforward extension to the case where the potential has
regular singularities at both end points. The proof does not require any essentially
new idea; the details, however, are a bit tedious and are therefore left to the reader,
cf. Remark \plref{r:Concluding}.

The case $\nu=0$ and $\theta_0>0$, which is not covered by Theorem
\plref{thm:Main}, requires specific analysis of unusual singular phenomena 
in the trace expansion of $H$, 
as observed first by  Falomir, Muschietti, Pisani, and Seeley \cite{FMPS:UPZ}; 
see also the nice elaboration 
by Kirsten, Loya and Park in \cite{KLP:UPR}. 
The discussion of the zeta-determinant in this case
therefore requires another publication.


To outline the proof of Theorem \plref{thm:Main} we first observe that if $D_1, D_2$ are closed operators in a Hilbert space
then $\spec D_1D_2\cup\{0\}=\spec D_2D_1\cup \{0\}$ and,  even more, 
nonzero eigenvalues of $D_1D_2$ and $D_2D_1$ have the same multiplicity.
Hence if both $D_1D_2$ and $D_2D_1$ satisfy the general assumptions of Section \plref{ss:ZetaDeterminant} then
for $z\in \C$
\begin{equation}
     \detz(D_1D_2+z)= z^d \detz(D_2D_1+z),\quad d:=\dim\ker D_1D_2-\dim\ker D_2D_1.
\end{equation}
We will show in Proposition \plref{p:FactorH} that $H(\theta_0,\theta_1)$ can always be written in the form
\begin{equation}\label{eq:FactorH}
        H(\theta_0,\theta_1)=D_1D_2 +W
\end{equation}
with a compactly supported $L^2$-potential $W$ and $D_1, D_2$ suitable closed extensions of the operators
\begin{equation}
          d_1=\frac{d}{dx}+\frac{\go'}{\go},\quad d_2=-\frac{d}{dx}+\frac{\go'}{\go}
\end{equation}
with a certain function $\go$ which is singular at $0$; its properties will be described in detail in the text.
The crucial point is that for the interesting case $\theta_0>0$ one can choose $D_1, D_2$ in such a way
that $D_2D_1$ also is an operator to which Theorem \plref{thm:ResolventExpansion} applies and such that
the boundary condition at $0$ is the Friedrichs extension. 
The Friedrichs extension at $0$ is much better behaved and can be treated for our class of operators
basically as in \cite{Les:DRS}. The proof is completed then by employing a variation result for 
the behavior of the zeta-determinant under variation of the potential $W$ (Theorem \plref{thm:PotentialVariation}).

\section{The fundamental system of a regular singular equation --  B\^{o}cher's Theorem}\label{sec:FunSys}

Consider the following regular singular model operator 
\begin{align}\label{1-a}
l_\nu:=-\frac{d^2}{dx^2}+\frac{\nu^2-1/4}{x^2}, \quad \nu\in \C,
\end{align}
acting on $C^{\infty}_0(0,1)\subset L^2[0,1]$.  $\nu$ is a complex number for which without loss of generality
we may assume $\Re\nu\geq 0$.

We are interested in perturbations of the form
\begin{align}\label{2-a}
H:=l_\nu+X\ii V,
\end{align}
with $V\in L^1_\loc(0,1)$ and $X$ denoting the function $X(x)=x$. In this section we are concerned with the 
description of the asymptotic behavior as $x\to 0$ of a fundamental system of solutions to the equation $Hf=0$. 

If $V$ is \emph{analytic}, then the classical theory of ordinary differential equations with regular singularities, 
cf. e.g. \cite[Chap. 5]{CodLev:TOD}, applies and the characteristic exponents of the regular singular point at 
$x=0$ are given by $$\nu_1=\nu + 1/2, \quad \nu_2=-\nu +1/2.$$

Furthermore, there is a fundamental system of solutions to $Hf=0$ of the form
\begin{align}\label{6-a}
f_1(x)=x^{\nu_1}\widetilde{f}_1(x), \quad f_2(x)=
\begin{cases}-\frac{1}{2\nu}x^{\nu_2}\widetilde{f}_2(x), &\textup{if } \nu \neq 0, \\
 \sqrt{x}\log (x)\; \widetilde{f}_2(x),  & \textup{if }\nu=0,
\end{cases} 
\end{align}
where $\widetilde{f}_j, j=1,2,$ are analytic functions with $\widetilde{f}_j(0)=1$. 
The normalization of solutions is chosen so that 
\begin{align}\label{7-a}
W(f_1,f_2)=f_1 f_2' -f_1' f_2=1.
\end{align}

It is less known that already M. B\^{o}cher \cite{Boc:RSP} investigated regular singular 
points of ordinary differential equations with non-analytic coefficients. 
For Bessel operators with $L^2$ potentials a thorough analysis of the
fundamental system of solutions was made e.g. by Carlson \cite{Car:BLT}.
B\^{o}cher's result reads as follows.
\pagebreak[3]
\begin{theorem}\label{5-a-t} \textup{[M. B\^ocher]}
Let
$$H=-\frac{d^2}{dx^2}+\frac{\nu^2-1/4}{x^2}+\frac{1}{x}V(x), $$
where $\nu\in \C, \Re\nu\geq 0,$ and $V\in \llog\sVnu$.
Then the differential equation $Hg=0$ has a fundamental system of solutions $g_1,g_2$, such that 
\begin{align}
g_1(x) =& \ x^{\nu_1}\widetilde{g}_1(x),\label{40} \\
g_2(x) =& \, \begin{cases}-\frac{1}{2\nu}x^{\nu_2}\widetilde{g}_2(x),  &
 \textup{if }\nu \neq 0, \\
                     \sqrt{x}\log (x) \widetilde{g}_2(x), & \textup{if }\nu=0, 
             \end{cases} \label{40-a} 
\end{align}
where $\widetilde{g}_j\in C[0,1], \ \widetilde{g}_j(0)=1$ for $j=1,2$. 
 
Furthermore, 
\begin{align}
g'_1(x) =& \ \nu_1x^{\nu_1-1}h_1(x), \label{41} \\
g'_2(x) =& \, \begin{cases}-\frac{\nu_2}{2\nu}x^{\nu_2-1}h_2(x),  &\textup{if }\nu \neq 0, \\
            \frac{1}{2\sqrt{x}}\log (x) h_2(x), & \textup{if }\nu=0, 
            \end{cases} \label{41-a} 
\end{align}
where $h_j\in C[0,1], \ h_j(0)=1$ for $j=1,2$.\bigskip

Finally, with these normalizations
$$W(g_1,g_2)=g_1g_2'-g_1'g_2=1.$$
\end{theorem}

\begin{remark}
In the case $\nu =0$ the theorem as stated is slightly more general than
\cite{Boc:RSP}, where $V\log^2 \in L^1[0,1]$ is assumed. Moreover, note that
the conditions on the potential $V$ in the theorem are satisfied whenever
$V\in L^p[0,1], p>1$,
or more generally $V\in\sV_\nu$.
\end{remark}

We briefly sketch a proof of Theorem \ref{5-a-t} in modern language. Being self--contained is not
the only reason for presenting the proof in some detail: the method of proof will allow a more
precise analysis of the regularity properties of $\tilde g_j$ (see Prop. \plref{les-3} and Prop. \plref{les-4}
below) which will be needed later on. Furthermore, the method will be needed for deriving the variation formula
for the zeta-determinant under variation of the potential (Section \plref{sec:VarPot}). 

\begin{proof}
The regular singular operator $l_\nu$ has the following fundamental system of solutions to $l_\nu f=0$:
\begin{align*}
f_1(x)=x^{\nu_1}, \quad f_2(x)=
   \begin{cases}-\frac{1}{2\nu}x^{\nu_2}, & \textup{if }\nu \neq 0, \\ 
       \sqrt{x}\log x,  & \textup{if }\nu=0.
   \end{cases}  
\end{align*}

For the Wronskian we have $W(f_1,f_2)=f_1 f_2'-f_1'f_2=1$. For a solution to $Hg=0$ we make the ansatz
$$g_1(x)=f_1(x)+x^{\nu_1}\phi(x)=x^{\nu_1}(1+\phi(x)).$$

Plugging this ansatz into the ordinary differential equation $Hg=0$ yields for $\psi(x)=x^{\nu_1}\phi(x)$
\begin{align}\label{20-a}
-\psi''(x)+\frac{\nu^2-1/4}{x^2}\psi(x)=-\frac{1}{x}V(x)[f_1(x)+\psi(x)],
\end{align}
thus 
\begin{equation}
\begin{split}
\phi(x)=& - x^{-\nu_1}f_1(x)\int_0^xf_2(y)\frac{1}{y}V(y)y^{\nu_1}[1+\phi(y)]dy \\
        &+x^{-\nu_1}f_2(x)\int_0^xf_1(y)\frac{1}{y}V(y)y^{\nu_1}[1+\phi(y)]dy\\
       =:&(K_\nu V \one) (x)+(K_\nu V \phi )(x), \label{22-a}
\end{split}
\end{equation}
where $K_{\nu}$ is the Volterra operator with the kernel 
\begin{align}
k_{\nu}(x,y)=&\frac{1}{2\nu}\bigl(1-x^{-2\nu}y^{2\nu}\bigr),   &y\leq x,
\textup{ if }  \nu\neq 0, \label{eq:VolterraOperatorA}\\
k_{0}(x,y)  =&-\log (y) + \log (x),                            &y\leq x,
\textup{ if }  \nu = 0.   \label{eq:VolterraOperatorB}
\end{align}
We view $K_\nu V$ as an operator on the Banach space $C[0,1]$. Indeed, for any $\phi \in C[0,1]$ one easily checks
\begin{align}
|K_\nu V\phi(x)| &\leq   \frac{1}{|\nu|} \int_0^x |V(y)||\phi(y)|dy,    &\textup{if }\nu\neq 0, \label{11-a} \\
|K_0 V \phi(x)|  &\leq  2\int_0^x |\log (y)||V(y)| |\phi(y)|dy,          &\textup{if }\nu=0. \label{16-a}
\end{align}
From \eqref{11-a} and \eqref{16-a} one infers by induction
\begin{align}\label{14-a}
|(K_\nu V)^n\phi(x)| &\leq \frac{1}{|\nu|^nn!}\|\phi\|_{\infty,[0,x]} \Bigl(\int_0^x |V(y)|dy\Bigr)^n,   &\textup{if }\nu\neq 0, \\
\label{18-a}
|(K_0 V)^n\phi(x)| &\leq \frac{2^n}{n!}\|\phi\|_{\infty,[0,x]} \Bigl(\int_0^x |V(y)\log y|dy\Bigr)^n,    &\textup{if }\nu =0.
\end{align}
Hence for any $\nu\in\C, \Re\nu\geq 0,$ the Volterra operator $K_\nu V$ is a bounded operator on $C[0,1]$ with spectral radius zero. Consequently the equation \eqref{22-a} has a unique solution in $C[0,1]$ given by
\begin{align}\label{24-a}
\phi=(I-K_\nu V)^{-1}K_\nu V\one.
\end{align}
Moreover, by \eqref{14-a} and \eqref{18-a} one has
\begin{align}
|\phi(x)|\leq \begin{cases} C_1\int_0^x|V(y)|dy, & \textup{if }\nu\neq0, \\
       C_2\int_0^x|V(y)\log y|dy, & \textup{if }\nu=0,  
              \end{cases} \label{32-a} 
\end{align}
for some constants $C_1,C_2,$ not depending on $V$. This proves that 
\[g_1(x)=x^{\nu_1}(1+\phi(x))=x^{\nu_1}\widetilde{g}_1(x),
\]
is indeed a non-trivial solution to $Hg=0$ with $\widetilde{g}_1\in C[0,1]$ and $\widetilde{g}_1(0)=1$. 
To see \eqref{41}, note that by \eqref{22-a} $\phi$ is absolutely continuous in $(0,1)$ with its derivative given by
\begin{equation}\label{eq:ML1003055}
\phi'(x)=x^{-2\nu-1}\int_0^xy^{2\nu}V(y)(1+\phi(y))dy, \textup{ for all } \Re \nu\ge 0. 
\end{equation}
This implies $$|\phi'(x)|\leq \frac{C}{x}\int_0^x|V(y)|dy$$ and thus \eqref{41} and the claims about $g_1$ are proved. \bigskip

The second solution $g_2$ can now be constructed as usual by putting near $x=0$
\begin{align}\label{35-a}
g_2(x)=C(x)g_1(x),
\end{align}
where 
\begin{equation}
C(x)=-\int_x^{x_0}g^{-2}_1(y)dy=\begin{cases}
    -\frac{1}{2\nu}x^{-2\nu_1+1}\widetilde{C}(x), &\textup{if } \nu\neq 0,  \\ 
    \log(x) \widetilde{C}(x), & \textup{if }\nu=0.
   \end{cases} \label{38-a}
\end{equation}
$\widetilde{C}$ is continuous over $[0,x_0)$ and $x_0\in (0,1]$ is chosen so that $\widetilde{g}_1(y)\neq 0$ 
for $0<y\leq x_0$. Such an $x_0$ exists, since $\widetilde{g}_1(0)=1$ and $\widetilde{g}_1\in C[0,1]$. 
It is then straightforward to check that $g_2$ extends to a solution to $lg=0$ on $(0,1]$ which has the claimed properties. 
\end{proof}

Now we come to the aforementioned improvement of the regularity properties of $\tilde g_j(x)$ as $x\to 0$.

\begin{lemma}\label{l:ML1003101} Let $f\in AC^2(0,1)$ with $f', X f''\in L^1[0,1]$. Then $(X f')(0)=0$.
This holds in particular for $f\in\sA$ (cf. Def. \plref{def:FunctionSpaces}).
\end{lemma}
\begin{proof} By assumption the function $F(x):= x f'(x)-\int_0^x h(s) ds, h:=(X f')',$
is locally absolutely continuous and $F'=0$, hence $F=(X f')(0)=:c$ is constant.
Thus
\begin{equation}
     f'(x)=\frac cx + \frac 1x \int_0^x h(s)ds.
\end{equation}
By assumption we have $f', h\in L^1[0,1]$. Thus 
\begin{equation}\label{eq:ML100206-2}
     \begin{split}
          f(1)-f(x)&=\int_x^1 f'(s)ds\\
                  &= -c \log x - \log x \int_0^x h(s)ds- \int_x^1 h(s) \log s \; ds\\
                  &= -c \log x + o(\log x),\quad \text{ as } x\to 0,
\end{split}
\end{equation}
since for $0<\delta<1$ we have $|\int_x^1 h(s)\log s\, ds|\le C_\delta+|\log x|\int_0^\delta |h|.$
Because the left hand side of \eqref{eq:ML100206-2} is bounded it follows that $c=0$.

The last claim follows, since $f\in\sA$ implies $f',(X f')'\in L^1[0,1]$.
\end{proof}

\begin{lemma}\label{les-2}
Let $\A\in \C$ and $\rho\in\llog \sV_{\Re \ga}$.
Put 
\begin{equation}
f(x):=\begin{cases}
x^{-\A-1}\int\limits_0^xy^{\A}\rho(y)dy, & \textup{if }\Re(\A)\geq 0, \\
-x^{-\A-1}\int\limits_x^1 y^{\A}\rho(y)dy, & \textup{if }\Re(\A)< 0.
\end{cases}
\end{equation}
Then we have 
$$f\in L^1[0,1]\cap AC(0,1), \quad X  f'\in L^1[0,1], \quad (X f)(0)=0.$$
If $\rho\in\sV_{\Re \ga}$ then $f, X f' \in \llog\ii L^1[0,1],$ that is
$\int_0^\cdot f \in \sA$.
\end{lemma}
\begin{proof}
Integration by parts shows easily that $f\in L^1[0,1]$ (resp. $f\in \llog\ii L^1[0,1]$ if $\rho \in \sV_{\Re \ga}$). 
Moreover, clearly $f$ and hence also $X  f$ are both locally absolutely continuous in the interval $(0,1]$. 
Furthermore, we have 
\begin{align*}
X f'=(X f)'-f=-(\A+1) f +\rho \in L^1[0,1] \\
\text{ (resp. } \in \llog\ii L^1[0,1]  \textup{ if }  \rho \in \sV_{\Re \ga}).
\end{align*}
$(X f)(0)=0$ follows from Lemma \ref{l:ML1003101} applied to $\int_0^\cdot f$.
\end{proof}


\begin{prop}\label{les-3}
In the setup and notation of Theorem \plref{5-a-t} we have for $j=1,2$
$$\tilde{g}_j\in AC[0,1], \ X  \widetilde{g}_j '\in AC[0,1], \ (X \widetilde{g}_j ')(0)=0.$$
\end{prop}
\begin{proof}
We have for the first fundamental solution
$$g_1(x)=x^{\nu_1}\widetilde{g}_1=x^{\nu_1}(1+\phi(x)),$$
where by $\phi$ is given by \eqref{22-a}.
The claim about $\tilde g_1$ now follows from Lemma \ref{les-2} and the explicit
form of the derivative \eqref{eq:ML1003055}.
%
To prove the claim for $\widetilde{g}_2$, recall that for some $x_0\in (0,1]$, such that $g_1(y)\neq 0$ for $0<y\leq x_0$, 
the second fundamental solution is given by 
\[
g_2(x)=-g_1(x)\int_x^{x_0}g_1(y)^{-2}dy.
\]
If $\nu \neq 0$, then
\begin{equation}\label{les-nr3}
\begin{split}
\widetilde{g}_2(x)&=-2\nu x^{\nu-1/2}g_2(x)\\
                  &=2\nu \widetilde{g}_1(x)x^{2\nu}\int_x^{x_0}y^{-2\nu-1}\widetilde{g}_1(y)^{-2}dy=:\widetilde{g}_1(x)f(x).
\end{split}
\end{equation}
In view of the statement being proved for $\widetilde{g}_1$ before and since the claimed properties are preserved
under multiplication, it suffices to prove the claim for $f$. Integration by parts gives
\begin{align}\label{les-nr2}
f(x)&=-x^{2\nu}\left.(y^{-2\nu}\widetilde{g}_1(y)^{-2})\right|_x^{x_0}
                  + x^{2\nu}\int_x^{x_0}y^{-2\nu}\rho(y)dy\nonumber\\
    &=c(x_0)x^{2\nu} + \widetilde{g}_1(x)^{-2} +
    x^{2\nu}\int_x^{x_0}y^{-2\nu}\rho(y)dy,
\end{align}
where $\rho =(\widetilde{g}_1^{-2})'\in L^1[0,x_0]$.

The first two summands are a priori in $AC[0,x_0]$. The last one is $AC[0,x_0]$ by Lemma \ref{les-2}. 
Furthermore, from the definition we infer 
\[
X  f'=2\nu f-2\nu \widetilde{g}_1^{-2}\in AC[0,x_0].
\]
$(X f')(0)=0$ then follows from Lemma \plref{l:ML1003101}.

Clearly, $\widetilde{g}_2$ and $X \widetilde{g}_2'$ are locally absolutely continuous in the whole interval $(0,1]$ and hence the claim is
proved for $\tilde g_2$ and $\nu\not=0$. 

Finally, for $\nu=0$ we have
\begin{align*}
\widetilde{g}_2(x)=&-\frac{1}{\sqrt{x}\log x} g_1(x)\int_x^{x_0}y^{-1}\widetilde{g}_1(y)^{-2}dy \\
=&-\frac{1}{\log x} \widetilde{g}_1(x)\int_x^{x_0}y^{-1}\widetilde{g}_1(y)^{-2}dy =:\widetilde{g}_1(x)f(x).
\end{align*}

Again, it suffices to prove the claim for $f$. We compute with 
$\rho:=(\widetilde{g}_1^{-2})'\in L^1[0,x_0]$ as before 
\begin{align*}
f(x)=&-\frac{1}{\log x} \int_x^{x_0}y^{-1}\widetilde{g}_1(y)^{-2}dy \\
=&-\frac{1}{\log x}\left.\log (y)\widetilde{g}_1(y)^{-2}\right|_x^{x_0}+ \frac{1}{\log x}\int_x^{x_0}\log (y) \rho(y) dy \\
=&\frac{c(x_0)}{\log x} + \widetilde{g}_1(x)^{-2} + \frac{1}{\log x}\int_x^{x_0}\log (y) \rho(y) dy.
\end{align*}

From this one checks that $f\in AC[0,x_0]$ and hence $\tilde g_2\in AC[0,1]$. 
Furthermore
\[
xf'(x)=-\frac{c(x_0)}{\log^2 x}-\frac{1}{\log ^2(x)}\int_x^{x_0}\log(y)\rho(y)dy,
\]
and differentiating this again shows $(X f')'\in L^1[0,1]$.
The remaining claims now follow as in the case $\nu\not=0$.
%
\end{proof}

Finally we prove the following refinement of the properties of $\widetilde{g}_j, j=1,2$, which will be crucial for the rest of the paper. 
\begin{prop}\label{les-4}
Let $V\in\sV_{\Re \nu}$. Then, in the notation of Theorem \plref{5-a-t} 
we have for $j=1,2$ 
\[
\tilde g_j'\, \llog, X\,  \tilde g_j''\,  \llog \in L^1[0,1],
\]
i.e. $\tilde g_j\in\sA$.
\end{prop}
\begin{proof} We prove the result only for $\nu\not=0$ and leave the case
$\nu=0$ to the reader. The result for $\nu=0$ will not be used in the rest of the paper.

Recall that $g_1(x)=x^{\nu_1}\tilde g_1(x)=x^{\nu_1}(1+\phi(x)),$
where $\phi\in AC[0,1]$ (Prop. \plref{les-3}) is given by \eqref{22-a} and observe that by \eqref{eq:ML1003055}
we have
\begin{equation}
\phi'(x)=x^{-2\nu-1}\int\limits_0^xy^{2\nu} r(y) \, dy,\label{eq:ML1003121}
\end{equation}
with $r:=V\cdot (1+\phi)$, $r\llog\in L^1[0,1]$ since $V\in\sV_{\Re\nu}$. 
The claims about $\tilde g_1$ now follow from \eqref{eq:ML1003121} and Lemma \plref{les-2}.

Recall from \eqref{les-nr3} $\tilde g_2(x)=\tilde g_1(x) f(x)$ with $f$ given by \eqref{les-nr2}. Differentiating the latter
we find
\begin{equation}
f'(x)=\widetilde{c}(x_0)x^{2\nu-1} +2\nu x^{2\nu-1}\int_x^{x_0}y^{-2\nu}(\widetilde{g}_1^{-2})'(y)dy.\label{eq:ML1003122}
\end{equation}
The first summand is still in $L^1[0,x_0]$ after multiplying by $\llog$. 
For the second summand note that $(\tilde{g}_1^{-2})'\,
\llog=-2\tilde{g}_1^{-3} \tilde{g}_1'\, \llog \in L^1[0,x_0]$ by the
inclusion $\tilde g_1'\, \llog \in L^1[0,1]$, proved above.
Hence we can apply Lemma \ref{les-2} to the second summand to conclude $f'\, \llog \in L^1[0,x_0]$.
Differentiating \eqref{eq:ML1003122} we infer similarly $X f''\,  \llog\in L^1[0,x_0]$.
Then one easily checks the claimed properties for the product $\tilde g_2=\tilde g_1 f$.
Since $\tilde g_2'\,  \llog$ and $X\, \tilde g_2''\,  \llog$ are
locally integrable in the interval $(0,1]$ we reach the conclusion.
\end{proof}

\section{The maximal domain of regular singular operators}\label{s-a}
\label{sec:MaxDom}

We continue in the notation of the preceding section and consider
the regular singular Sturm-Liouville operator $H$ with the fundamental 
system $(g_1,g_2)$ of solutions to the differential equation $Hg=0$ (cf. Theorem \plref{5-a-t}). 
We will freely use the notation introduced in Section \plref{ss:SepBouCon}.

We have the following characterization of the maximal domain of $H$, 
compare \cite{BruSee:ITF} and \cite{Che:SGS} and the basic discussion of the second author 
in \cite[Proposition 2.10]{Ver:ZDR}. Note that it holds under a slightly weaker
assumption on the potential $\bigl(V\in \lsV_\nu\bigr)$ than the one imposed in the rest of the paper.

\begin{theorem}\label{thm:LaplacianMaximalGeneral}
Let $l_\nu$ be the operator \eqref{1-a} and let $H=l_\nu+X^{-1} V$ with 
$V\in \lsV_\nu$, $\Re\nu\ge 0$, and let $g_1,g_2$ be the fundamental system to
$Hg=0$ of Theorem \plref{5-a-t}.
Let $f$ be a solution of the ordinary differential equation
\begin{align}\label{M91}
Hf= -&f''+qf=g\in L^2[0,1], 
\\ &q(x)=\frac{\nu^2-1/4}{x^2}+\frac{1}{x}V(x).\nonumber
\end{align}
Then $f\in AC^2(0,1]$ and 
\begin{align}\label{M92}
f(x)=c_1(f)g_1(x)+c_2(f)g_2(x)+\widetilde{f}(x),
\end{align}
for some constants $c_j(f), j=1,2,$ depending only of $f$, 
\begin{align}\label{M93}
\widetilde{f}(x)=O\bigl(x^{3/2}\log (x)\bigr), \
\widetilde{f}'(x)=O\bigl(x^{1/2}\log(x)\bigr), \ x\to 0+.
\end{align}
\end{theorem}
\begin{remark}\label{rem:ML20100430}
We emphasize that the solution $g_1$ is completely determined by the equation
\eqref{40} and therefore canonical. However this is not so for $g_2$. Surely, any function
$g_2+\lambda g_1$ also satisfies \eqref{40-a}. We mention this because as a consequence
the functional $c_2$ (!) is canonically given while $c_1$ depends on the choice of $g_2$.
\end{remark}

\begin{proof} 
We first note that it is well--known that solutions to linear differential equations with 
$L^1_\loc$ coefficients are locally absolutely continuous. 
Therefore a solution $f$ to \eqref{M91} is absolutely continuous in the interval $(0,1]$ and 
from $f''=g-qf$ one then infers that $f'$ is also absolutely continuous in $(0,1]$. 

For $x_0\in \{0,1\}$ 
\begin{align}\label{M110}
\widetilde{f}(x)=g_1(x)\int_{x_0}^x g_2(y)g(y)dy - g_2(x)\int_{0}^xg_1(y)g(y)dy
\end{align} 
is a solution to \eqref{M91}; note $W(g_1,g_2)=1$.
Depending on $\nu$ we will choose $x_0$ such that \eqref{M93} is satisfied. We first note that by applying the Cauchy-Schwarz inequality
\begin{equation}
\begin{split}
\bigl|g_2(x)\int_{0}^x g_1(y)g(y)dy\bigr| &\leq |g_2(x)|\Bigl(\int_0^x|g_1(y)|^2dy\Bigr)^{1/2}\|g\|_{L^2} \\
                                          &=\begin{cases}
                                           O\bigl(x^{3/2} \bigr), & \textup{if }\nu \neq 0, \\
                                                O\bigl(x^{3/2}\log(x)\bigr), & \textup{if }\nu = 0,
                                            \end{cases}\quad  x\to 0+.\label{M111}
\end{split}
\end{equation}
Furthermore, if $\Re\nu\geq 1$, we put $x_0=1$ and find
\begin{equation}
\begin{split}
\bigl|g_1(x)\bigr| 
\Bigl|\int_{x}^1g_2(y)g(y)dy\Bigr| &\leq C x^{\Re\nu+1/2}\Bigl(\int_x^1y^{-2\Re\nu+1}dy\Bigr)^{1/2}\|g\|_{L^2} \\
               &=\begin{cases}
                 O(x^{3/2}|\log (x)|^{1/2}), & \textup{if }\Re\nu =1, \\ 
                 O(x^{3/2}), & \textup{if }\Re\nu >1,
                  \end{cases}\quad  x\to 0+. \label{M112}
\end{split}
\end{equation}
Finally, if $0\leq \Re\nu<1$, we put $x_0=0$ and estimate
\begin{equation}
\begin{split}
\bigl|g_1(x)\bigr| \Bigl|\int_{0}^x g_2(y)g(y)dy\Bigr| 
            &\leq C x^{\Re\nu+1/2}\Bigl( \int_0^x |g_2(y)|^2 dy\Bigr)^{1/2}\|g\|_{L^2} \\
            &=\begin{cases} 
                 O\bigl(x^{3/2}\bigr), & \textup{if }\nu\neq 0, \\ 
                 O\bigl(x^{3/2}\log(x)\bigr), & \textup{if }\nu=0,
               \end{cases} \quad  x\to 0+. \label{M121}
\end{split}
\end{equation}
This proves the estimates for $\widetilde{f}(x)$. Differentiating \eqref{M110} we find 
\begin{align*}
\widetilde{f}'(x)=g'_1(x)\int_{x_0}^x g_2(y)g(y)dy - g'_2(x)\int_{0}^xg_1(y)g(y)dy,
\end{align*}
and \eqref{41}, \eqref{41-a} together with \eqref{M111},\eqref{M112} and \eqref{M121} immediately give the
claimed estimate for $\widetilde{f}'$. 
Thus $\widetilde{f}$ is a solution to \eqref{M91} satisfying \eqref{M93}. \eqref{M92} is now clear.
\end{proof}

\begin{remark}
The above proof shows that for $\nu \neq 0, \Re\nu\neq 1$, the estimate \eqref{M93} can actually be replaced by 
\begin{align}\label{M93-b}
\widetilde{f}(x)=O(x^{3/2}), \ \widetilde{f}'(x)=O(x^{1/2}), \ x\to 0+\, ,
\end{align}
and for $\Re\nu=1$ by 
\begin{align}\label{M93-c}
\widetilde{f}(x)=O(x^{3/2}|\log (x)|^{1/2}), \ \widetilde{f}'(x)=O(x^{1/2}|\log(x)|^{1/2}), \ x\to 0+.
\end{align}
\end{remark}

\begin{cor}\label{cor:MaxDom} Under the assumptions of Theorem \plref{thm:LaplacianMaximalGeneral}
there are continuous linear functionals $c_j, j=1,2$ on $\domHmax$ such that for $f\in\domHmax$
\begin{equation}\label{eq:ML100125-3}
     f=c_1(f)g_1 +c_2(f)g_2+\tilde f,
\end{equation}
with $\tilde f\in \domL(H)$ (cf. Section \plref{ss:SepBouCon}). 
Let $c_j^t$ be the corresponding functionals for $H_0^t$. Then we have for 
$f\in\domHmax, g\in\dom(H^t_{\max})$
\begin{equation}\label{eq:ML100125-4}
        \scalar{\Hmax f}{g}- \scalar{f}{H^t_{\max}g} =-\ovl{f'(1)} g(1)+ \ovl{f(1)} g'(1)
            + \ovl{c_2(f)} c_1^t(g) -\ovl{c_1(f)} c_2^t(g).
\end{equation}
Finally, $0$ is in the limit point case for $H$ if and only if $\Re\nu\ge 1$. In this
case $c_2=0$.
\end{cor}
\begin{proof} \eqref{eq:ML100125-4} follows easily from \eqref{M92}, \eqref{M93} and the 
 Lagrange formula.
The formulas \eqref{eq:ML100125-4} and \eqref{M93} show that $\tilde
f\in\domL(H)$; the latter was defined in Section \ref{ss:SepBouCon}.
Now, it follows from \eqref{M92} that the quotient space $\domHmax/\domL(H)$ is spanned
by $g_1+\domL(H), g_2+\domL(H)$ if $g_2\in L^2[0,1]$ and by $g_1+\domL(H)$ if $g_2\not\in L^2[0,1]$.
This implies the continuity of the functionals $c_1,c_2$ on $\domHmax$. 
Finally, $c_2=0$ if and only if $g_2\not\in L^2[0,1]$. The latter is equivalent to $\Re \nu\ge 1$
and the claim is proved.
\end{proof}

As already outlined in Section \plref{ss:SepBouCon} we obtain a closed extension
of $H_0$ with separated boundary conditions by choosing boundary operators $B_{j,\theta_j}$
\begin{equation}\label{eq:BouOp-3}
\begin{split}
       B_{0,\theta_0}f&:= \sin\theta_0\cdot c_1(f)+ \cos\theta_0\cdot c_2(f), \\
       B_{1,\theta_1}f&:= \sin\theta_1\cdot f'(1)+ \cos\theta_1\cdot f(1),
\end{split} \quad (\theta_0,\theta_1)\in [0,\pi)^2.
\end{equation}
$B_{0,\theta_0}$ depends on the choice of a fundamental system, cf. Remark \plref{rem:ML20100430}.
To treat the limit point and limit circle cases at $0$ in a unified way we
call a pair of boundary operators \emph{admissible} if $\theta_1\in [0,\pi)$ and 
either ($\theta_0=0$ and $\Re\nu\ge 1$) or ($\theta_0\in [0,\pi)$ and $0\le
\Re \nu<1$).

Given an admissible pair $B_{j,\theta_j}, j=0,1$, of boundary operators we denote by
$H(\theta_0,\theta_1)$ the closed extension of $H_0$ with domain
\begin{equation}\label{l-t}
H(\theta_0,\theta_1):=\bigsetdef{f\in \domHmax}{ B_{j,\T_j}f=0, \, j=0,1}. 
\end{equation}
If $\nu\in \R$ and $V$ is real valued then $H_0$ is symmetric and $H(\T_0,\T_1)$ is self--adjoint. 
If $\Re\nu \ge 1$ then all self--adjoint extensions are obtained in this way. If $\Re\nu <1$ there also
exist self--adjoint extensions with non--separated boundary conditions. These extensions will not
be studied in this paper. 

\subsection{Factorizable operators}\label{ss:FactorOperator}

Next we investigate when $H$ can be factorized as $d_1 d_2$ with
$d_j=\pm \frac{d}{dx} + \frac{\go'}{\go}$. For simplicity
we confine ourselves to the case $\nu\not=0$. Clearly,
with some modifications one has similar results for $\nu=0$.

We have seen in Proposition \plref{les-4} that if $V\in\sV_{\Re \nu}$ and $\go$ is a solution to the
differential equation $H\go=0$ then $\go(x)=x^{\mu+1/2}\tilde\go(x)$
with $\tilde\go\in\sA$ and $\mu=\pm\nu$.

Conversely, let $\mu\in\C$ and $\tilde\go\in\sA$ with $\tilde\go(x)\not=0,
0\le x\le 1$, be given. Put 
\begin{equation}\label{eq:d1d2}
\begin{split}
    d_1&:=\frac{d}{dx} +\frac{\go'}{\go}=\frac{d}{dx}+\frac{\mu+1/2}{x}+
             \frac{\tilde\go'}{\tilde\go},\\
    d_2&:=-\frac{d}{dx}+\frac{\go'}{\go}.
\end{split}
\end{equation}
Then
\begin{align}
       H_{12}&:=d_1 d_2= -\frac{d^2}{dx^2}+\frac{\go''}{\go}\nonumber\\
             &= -\frac{d^2}{dx^2}+\frac{\mu^2-1/4}{x^2}+2\frac{\mu+1/2}{x}\frac{\tilde \go'}{\tilde\go}
                                          +\frac{\tilde \go''}{\tilde \go},\label{eq:H12}\\
       H_{21}&:=d_2 d_1=-\frac{d^2}{dx^2}+ 2\Bigl(\frac{\go'}{\go}\Bigr)^2-\frac{\go''}{\go},\nonumber\\
             &=-\frac{d^2}{dx^2}+\frac{(\mu+1)^2-1/4}{x^2}+2 \frac{\mu+1/2}{x} \frac{\tilde \go'}{\tilde\go}+
                 2\Bigl(\frac{\tilde \go'}{\tilde \go}\Bigr)^2-\frac{\tilde \go''}{\tilde \go},\label{eq:H21}
\end{align}
thus we have $H_{12}=l_\mu+X\ii V_{12}, H_{21}=l_{\mu+1}+X\ii V_{21}$
and using $\tilde \go\in\sA$ one directly checks $V_{12}, V_{21}\in\sV_\nu$ (see Definition \plref{def:FunctionSpaces}).
 
Hence for a given $V\in\sV_{\Re \nu}$ the operator $H=l_\nu+X\ii V$ can be factorized
in the form $H=d_1 d_2$ with $d_1,d_2$ as above if and only if there is a solution to the
homogeneous equation $H\go=0$ with $\go(x)\not=0$ for $0<x\le 1$. Indeed, if $\go$ solves $H\go=0$ and is 
nowhere vanishing in $(0,1]$, one directly verifies from the first line of \eqref{eq:H12} that $H_{12}=d_1 d_2$
coincides with $H$. Conversely, given a factorization of $H=d_1 d_2$ with $d_1,d_2$ as in \eqref{eq:d1d2} it is immediate 
that $d_2\go=0$ and thus $H\go=0$.
 
By Proposition \plref{les-4}, $\go$ is then of the form $\go(x)=x^{\pm \nu+1/2}\tilde\go(x)$
with $\go\in\sA$.

The problem is that such a nowhere vanishing solution does not necessarily exist.
However, we will be able to reduce the calculation
of the zeta-determinant to the calculation for
\emph{factorizable} operators. In fact, the main essence of the
Proposition \plref{p:FactorH} below is that although $H$ itself may not be
factorizable, it becomes factorizable after adding  a suitable
$L^2_\comp(0,1)$ potential. For such perturbations a variational
formula for the zeta-determinant will be established subsequently.

Next we investigate the separated boundary conditions for $d_j$.
For $d_j$ we can choose four possibly different closed extensions with
\emph{separated} boundary conditions: For $p,q\in\{a,r\}$ denote
by $d_{j,pq}$ the closed extension of $d_j$ with boundary condition
$p$ at the left end point and boundary condition $q$ at the right 
end point. Here, $r$ stands for the relative boundary condition and
$a$ for the absolute boundary condition. More concretely,
$d_{j,rr}=d_{j,\min}$ is the closure of $d_j$ on $\cinfz{(0,1)}$,
$d_{j,aa}=d_{j,\max}=(d_{j,\min}^t)^*$ is the maximal extension.
The domains of the mixed extensions can be characterized by
\begin{equation}
\begin{split}
   \dom(d_{j,ar})&=\bigsetdef{f\in\dom(d_{j,\max})}{f(1)=0},\\
   \dom(d_{j,ra})&=
       \bigsetdef{f\in\dom(d_{j,\max})}{f(x)=O(\sqrt{x}\; |\log(x)|^{1/2})\text{ as } t\to 0+}.
\end{split}
\end{equation}

For each choice of a closed extension $D_1$ of $d_1$ with boundary condition 
of the form $aa,rr,ar,ra$ we choose $D_2$ to be the closed extension with
dual boundary condition, i.e. $rr,aa,ra,ar$, for $d_2$. 
If $\go$ is real then this means
that $D_2=D_1^t$. We summarize case by case the corresponding boundary conditions
for $H_{12}, H_{21}$:

Note that $\go$ is a solution to the homogeneous differential equation $H_{12}g=0$. 
In the notation of Section \plref{ss:SepBouCon} we assume that 
\begin{equation}
     \go=\cos \vartheta_0 \cdot g_1 - \sin\vartheta_0\cdot g_2,
\end{equation}
with $0\le \vartheta_0<\pi$, thus $B_{0,\vartheta_0}\go=0$.
Moreover, we assume that $B_{1,\vartheta_1} \go=0$ for a $0<\vartheta_1\le\pi$.
We have to exclude $\vartheta_1=0$ because in that case $\go(1)=0$ and hence there would be
a regular singularity also at the right end point. But see Remark \plref{r:Concluding}.

For future reference we now list the 

\subsection{Separated boundary conditions for the factorized operator $D_1D_2$}\label{p:Cases}
\subsubsection*{Case I: $D_1=d_{1,rr}, D_2=d_{2,aa}$}  
$f\in\dom(D_1D_2)$ if and only if $f\in\dom(d_{2,\max})$ and $d_2f\in\dom(d_{1,\min})$. Thus one
checks that $D_1D_2=H_{12}(\vartheta_0,\vartheta_1)$ and
$D_2 D_1=H_{21}(0,0)$ is the Friedrichs extension of $H_{21}.$

\subsubsection*{Case II: $D_1=d_{1,ra}, D_2=d_{2,ar}$} Then
\[D_1 D_2=H_{12}(\vartheta_0,0), D_2 D_1=H_{21}(0,\pi-\vartheta_1).\]

\subsubsection*{Case III: $D_1=d_{1,ar}, D_2=d_{2,ra}$} Then
\[D_1 D_2=H_{12}(0,\vartheta_1), D_2 D_1=H_{21}(\pi-\vartheta_0,0).\]

\subsubsection*{Case IV: $D_1=d_{1,aa}, D_2=d_{2,rr}$} Then
\[
D_1 D_2=H_{12}(0,0), D_2 D_1=H_{21}(\pi-\vartheta_0,\pi-\vartheta_1).\]

We summarize the previous considerations in the following

\begin{prop}\label{p:FactorH} Let $V\in\sV_{\Re \nu}$ be given and let $H=l_\nu+X^{-1}V$ 
be the corresponding regular singular Sturm-Liouville operator. Suppose that we are given
admissible separated boundary conditions $B_{j,\theta_j}, j=0,1,$ for $H$. 
Then for $0<\vartheta_1<\pi$
there exists a function $\go(x)=x^{\mu+1/2}\tilde\go$
such that $\tilde\go\in\sA$ (see Def. \plref{def:FunctionSpaces}), 
$\tilde\go(x)\not=0$ for $0\le x\le 1$
and $B_{0,\theta_0}\go=0$, $B_{1,\vartheta_1}\go=0.$
Moreover, $H \go=0$ in a neighborhood of $0$ and in a neighborhood of $1$.

If $\theta_1=0$ we choose any $0<\vartheta_1<\pi$ and if $0<\theta_1<\pi$ we let
$\vartheta_1=\theta_1$. 

Putting $D_1=d_{1,ra}, D_2=d_{2,ar}$ if $0<\theta_0<\pi,\theta_1=0;$
$D_1=d_{1,rr}, D_2=d_{2,aa}$ if $0<\theta_0<\pi,\theta_1>0$; $D_1=d_{1,aa}, D_2=d_{2,rr}$ if
$\theta_0=\theta_1=0$ and $D_1=d_{1,ar}, D_2=d_{2,ra}$ if $\theta_0=0, \theta_1>0$ we have
\begin{equation}\label{eq:ML1003181}
       H(\theta_0,\theta_1)=D_1 D_2+W
\end{equation}
with $W\in L^2_{\comp}(0,1)$.

If $V$ is real then $\go$ can be chosen to be real.
\end{prop}
\begin{proof} We can certainly find an $\eps>0$ and solutions $\go_j, j=0,1$,
to the differential equation $Hg=0$ on the intervals $(0,\eps)$ resp.
$(1-\eps,1]$ such that $\go_0(x)\not=0$ for $0<x<\eps$, $\go_1(x)\not=0$
for $1-\eps\le x\le 1$ and $B_{0,\theta_0}\go_0=0$, $B_{1,\vartheta_1}\go_1=0$.
If $V$ is real we may choose $\go_j$ to be positive. In any case
we may choose a nowhere vanishing extension $\go$ to the whole interval
with the claimed regularity properties.

By construction $D_1D_2$ has the same boundary conditions as $H(\theta_0,\theta_1)$
and there are neighborhoods of $0,1$ resp. on which the potential of $D_1D_2$ coincides
with that of $H$, whence \eqref{eq:ML1003181}.
\end{proof}

\subsection{Comparison of Wronskians}

For a factorizable operator $H=D_1D_2$ and given admissible boundary conditions we
are now able to compare the Wronskians of normalized fundamental solutions of $D_1D_2$
and $D_2D_1$.
\begin{prop}\label{p:CommutatorWronskian} Let $\Re \nu\ge 0$ and let
$\go(x)=x^{-\nu+1/2}\tilde \go(x)$ with $\tilde \go\in\sA$, $\tilde \go(x)\not=0$ for $0\le x\le 1$.
Put $d_1,d_2$ as in \eqref{eq:d1d2} with $\mu=-\nu$. Assume that $B_{0,\vartheta_0}\go=0, B_{1,\vartheta_1}\go=0$
with admissible boundary conditions $B_{j,\vartheta_j}$, $0<\vartheta_j<\pi$ for $H_{12}=d_1d_2$. Choose
$D_1,D_2$ as in Proposition \plref{p:FactorH} with $\theta_0=\vartheta_0$ and
$\theta_1=0$ or $\theta_1=\vartheta_1$. Let $\varphi_z,\psi_z$ be normalized
solutions for $(D_2D_1+z)g=0$.

\textup{Case I: $0<\vartheta_1=\theta_1<\pi$}. Then $\tilde\varphi_z:=\frac{1}{2-2\nu} d_1\varphi$,
$\tilde\psi_z=-d_1\psi_z$ are normalized solutions for $(D_1D_2+z)g=0$. Furthermore, we have
\begin{equation}\label{eq:RelWronskianI}
   W(\tilde \psi_z,\tilde\varphi_z)=\frac{z}{2-2\nu}W(\psi_z,\varphi_z).
\end{equation}
$D_2D_1$ is invertible and the kernel of $D_1D_2$ is one--dimensional, hence $\spec D_1D_2=\spec D_2D_1\cup \{0\}$.

\textup{Case II: $\theta_1=0$}. Then $\tilde\varphi_z:=(2-2\nu)\ii d_1\varphi$,
$\tilde\psi_z=-z\ii d_1\psi_z$ are normalized solutions for $(D_1D_2+z)g=0$. 
Furthermore, we have
\begin{equation}\label{eq:RelWronskianII}
   W(\tilde \psi_z,\tilde\varphi_z)=\frac{1}{2-2\nu}W(\psi_z,\varphi_z).
\end{equation}
$D_2D_1, D_1D_2$ are both invertible and $\spec D_2D_1=\spec D_1D_2$.
\end{prop}
\begin{proof} In view of \eqref{eq:H12}, \eqref{eq:H21}
the characteristic exponents of $d_1d_2$ are $\pm\nu+1/2$ and the characteristic
exponents of $d_2d_1$ are $\pm(-\nu+1)+1/2$.
$d_1\varphi_z$ satisfies $(d_1d_2+z)d_1\varphi_z=d_1(d_2d_1+z)\varphi_z=0$. 
Furthermore, since $\varphi_z$ is normalized at $0$ for $D_2D_1+z$ we have 
in the notation of \eqref{sim}
$\varphi_z(x)\sim x^{3/2-\nu}$ as $x\to 0$ and using \eqref{eq:d1d2} we obtain
\begin{equation}
   d_1\varphi_z(x)\sim (3/2-\nu+1/2-\nu)x^{1/2-\nu}=(2-2\nu) x^{-\nu+1/2},
\end{equation}
hence $\frac{1}{2-2\nu}d_1\varphi_z$ is normalized at $0$ for $D_1D_2+z$ as claimed.
This applies to both cases I and II.

Now consider $d_1\psi_z$ which also solves $(d_1d_2+z)d_1\psi_z=0$.

\textup{Case I: $0<\vartheta_1=\theta_1<\pi$}. Then $D_2D_1=H_{21}(0,0)$ and hence
$\psi_z$ being normalized at $1$ means $\psi_z(1)=0, \psi_z'(1)=-1$. Then
$d_1\psi_z(1)=\psi_z'(1)+(\frac{\go'}{\go}\psi_z)(1)=-1$ and hence $-d_1\psi_z$ is normalized at $1$.

We now find for the Wronskian
\begin{align}
      W(&\tilde\psi_z,\tilde\varphi_z)=\frac{-1}{2-2\nu}W(d_1\psi_z,d_1\varphi_z)
            = \frac{-1}{2-2\nu}\bigl(d_1\psi_z (d_1\varphi_z)'-
            (d_1\psi_z)'d_1\varphi_z\bigr)\nonumber\\
           &=  \frac{1}{2-2\nu}\bigl(d_1\psi_z (d_2d_1\varphi_z)- (d_2d_1\psi_z)d_1\varphi_z\bigr)
           =  \frac{-z}{2-2\nu}\bigl((d_1\psi_z) \varphi_z- \psi_z
           d_1\varphi_z\bigr)\nonumber \\
           &=  \frac{z}{2-2\nu}W(\psi_z,\varphi_z).\label{eq:dWronskian}
\end{align}

\textup{Case II: $\theta_1=0$.} Then $D_2D_1=H_{21}(0,\pi-\vartheta_1)$ 
with $0<\pi-\vartheta_1<\pi$. Thus $\psi_z$ being normalized at $1$ means 
$\psi_z(1)=1$ and $B_{1,\pi-\vartheta_1}\psi_z=0$.
Then $\psi_z\in\dom(D_1=d_{1,ra})$ thus 
$0=d_1\psi_z(1)$. 
Hence $(d_1\psi_z)'(1)=-(d_2d_1\psi_z)(1)=z\psi_z(1)=z$. Thus $-z\ii d_1\psi_z$ is normalized at $1$ for
$H_{12}$. Now the same calculation as in \eqref{eq:dWronskian} yields \eqref{eq:RelWronskianII} and the Proposition
is proved.
\end{proof}




\section{The asymptotic expansion of the resolvent trace} \label{sec:ResolventExpansion} 

\subsection*{Standing assumptions} 
Let $l_\nu:=-\frac{d^2}{dx^2}+\frac{\nu^2-1/4}{x^2}$ be the regular singular model operator.
In this Section we assume $\nu$ to be real and non-negative.

\subsection{The Dirichlet condition at $0$}
Let $B_{1,\theta}=\sin\theta\cdot f'(1)+ \cos\theta\cdot f(1)$ be a boundary operator for the right end point and
let $L_\nu=L_\nu(0,\theta)$ be the closed extension of $l_\nu$ with domain
\begin{equation}\label{5-b}
\dom (L_\nu)=\bigsetdef{f \in \dom (l_{\nu,\max})}{ c_2(f)=0, \ B_{1,\theta}f =0}.
\end{equation}

The following perturbation result will be crucial for establishing the resolvent trace expansion
Theorem \plref{thm:ResolventExpansion} (cf. \cite[Lemma 3.1]{Les:DRS}).
\begin{prop}\label{p:PerturbationEstimate} 
Let $\nu\ge 0$. Then $L_\nu$ is self--adjoint and bounded below.

Let $W$ be a measurable function on $[0,1]$ such that 
\begin{equation}\label{eq:ML100128-1}
   \begin{cases}
         W^2 \in L^1[0,1], &\textup{if } \nu>0,\\
         W^2\, \llog\in L^1[0,1], &\textup{if } \nu=0.
\end{cases}
\end{equation}
Then for $z_0^2>\max \spec (-L_\nu)$ there is a constant $C(z_0)$ such that
for $z\in \R, z\ge z_0,$ we have for the Hilbert--Schmidt norms
\begin{equation}\label{eq:PerturbationEstimate}
\begin{split}
&\|x^{-1/2}W(L_\nu+z^2)^{-1/2}\|_{\HS}^2+\|(L_\nu+z^2)^{-1/2}x^{-1/2}W\|_{\HS}^2\\
   &\le C(z_0)\begin{cases}
     \Bigl( \frac 1z + \int_0^{1/z} |W(x)|^2 dx  + \frac 1z \int_{1/z}^1 \frac 1x |W(x)|^2 dx\Bigr), &\textup{if } \nu>0,\\
     \Bigl( \frac 1z + \int_0^{1/z} |W(x)|^2|\log(xz)| dx  + \frac 1z \int_{1/z}^1 \frac 1x |W(x)|^2 dx\Bigr), &\textup{if } \nu=0,
      \end{cases}\\
   &=: R(z).
\end{split}
\end{equation}
If \eqref{eq:ML100128-1} is replaced by $W^2\in\sV_\nu$, i.e.
\begin{equation}\label{eq:ML100128-2}
   \begin{cases}
         W^2\, \llog \in L^1[0,1], & \textup{if }\nu>0,\\
         W^2\, \llogv{2}\in L^1[0,1], & \textup{if }\nu=0,
\end{cases}
\end{equation}
then 
\begin{align}
&\lim_{z\to\infty} R(z)=0,\label{eq:IntCon-a}\\
&\int_{z_0}^\infty \frac 1z |R(z)| dz<\infty.\label{eq:IntCon-b}
\end{align}
\end{prop}

\begin{proof}
The boundary conditions for $L_\nu$ are separated and admissible. Therefore, $L_\nu$ is self--adjoint.
We will see below that the resolvent is Hilbert-Schmidt. Thus $L_\nu$ has a pure point spectrum.
An eigenfunction satisfying $L_\nu f=\gl^2 f,$ $\gl\in \R\cup i\R$, is therefore a multiple
of $\sqrt{x} J_\nu(\gl x)$, where $J_\nu$ denotes the Bessel function of order $\nu$
\cite{Wat:TTB}. From the known asymptotic behavior of the Bessel functions with
imaginary argument one deduces that $L_\nu$ has at most finitely many negative eigenvalues and
hence is bounded below.

The kernel $k_\nu(x,y;z)$ of the resolvent $(L_\nu+z^2)\ii$ is given in terms of the modified Bessel
functions $I_\nu, K_\nu$ 
\begin{equation}
  k_\nu(x,y;z) =\sqrt{xy} I_\nu(xz)\Bigl( K_\nu(yz)-\gb(z) I_\nu(yz)\Bigr),\quad x\le y,
\end{equation}
where $\beta(z)$ is determined by the requirement $B_{1,\theta}k(x,\cdot;z)=0$
(cf. \cite{BruSee:RSA}). One finds
\begin{equation}
\begin{split}
        \gb(z)&=\frac{(\cos\theta+\frac 12 \sin\theta) K_{\nu}(z)+z K_\nu'(z) \sin\theta}{%
                     (\cos\theta+\frac 12 \sin\theta) I_{\nu}(z)+z I_\nu'(z) \sin\theta}\\
              &=\frac{(\cos\theta+(\frac 12+\nu) \sin\theta) K_{\nu}(z)-z K_{\nu+1}(z)\sin\theta }{%
                      (\cos\theta+(\frac 12+\nu) \sin\theta) I_{\nu}(z)+z I_{\nu+1}(z)\sin\theta },
\end{split}
\end{equation}
where in the last equation we used the recursion relations \cite[3.71]{Wat:TTB}
\begin{equation}
    z I_\nu'(z)=z I_{\nu+1}(z) +\nu I_\nu(z),\quad z K_\nu'(z)=-z K_{\nu+1}+\nu K_\nu(z).
\end{equation}

Recall the following asymptotic relations for the modified Bessel functions \cite[7.23]{Wat:TTB}
\begin{equation}
\begin{split}
I_{\nu}(z)&=\frac{1}{\sqrt{2\pi z}}e^z\bigl(1+O(z^{-2})\bigr), \\ 
K_{\nu}(z)&=\sqrt{\frac{\pi}{2z}}e^{-z}\bigl(1+O(z^{-2})\bigr), 
\end{split}
\quad z\to \infty,  \label{eq:Bessel-large1}
\end{equation}
and  \cite[Sec. 3.7]{Wat:TTB}
\begin{equation}\label{eq:Bessel-small}
I_{\nu}(z)\sim \frac{1}{2^\nu \Gamma(\nu+1)} \,z^{\nu}, \quad
 K_{\nu}(z)\sim \begin{cases}
                2^{\nu-1}\Gamma(\nu) \,z^{-\nu}, &\textup{if } \nu \neq 0, \\
                 -\log z, &\textup{if } \nu=0,
                \end{cases}
\quad \textup{as} \ z\to 0;
\end{equation} 
for the notation $\sim$ see \eqref{sim}.
From the asymptotics as $z\to\infty$ one infers
\begin{equation}\label{eq:ML100128-3}
   \beta(z)=O (e^{-2z}), \quad z\to \infty.
\end{equation}

To prove the estimate \eqref{eq:PerturbationEstimate} we fix $z_0>\max\spec(-L_\nu)$ and find
for $z\ge z_0$
\begin{equation}\label{eq:ML1003051}
\begin{split}
 \|x^{-1/2}&W(L_\nu+z^2)^{-1/2}\|_{\HS}^2=\|(L_\nu+z^2)^{-1/2}x^{-1/2}W\|_{\HS}^2\\
      &=\Tr\bigl(x^{-1/2}W(L_\nu+z^2)\ii\ovl{W}x^{-1/2}\bigr)\\
      &= \int_0^1 x\ii |W(x)|^2 k_\nu(x,x;z) dx\\
      &= \int_0^1 |W(x)|^2 I_\nu(xz) K_\nu(xz) dx - \beta(z) \int_0^1 |W(x)|^2 |I_\nu(xz)|^2 dx.
\end{split}
\end{equation}
In the following $C$ denotes a generic constant depending only on $z_0$ and $\nu$. 
We split the integrals into an integration from $0$ to $1/z$ and from $1/z$ to $1$.
In the first regime \eqref{eq:Bessel-small} yields
\begin{equation}
           |I_\nu(xz)K_\nu(xz)|\le \begin{cases} C,&\textup{if } \nu \not=0,\\
                                           C |\log(xz)|,&\textup{if } \nu=0,
                                \end{cases}
\end{equation}
and $|I_\nu(xz)|\le C$. Thus, 
\begin{equation}\label{eq:ML100128-4}
\begin{split}
   \int_0^{1/z}& x\ii |W(x)|^2 k_\nu(x,x;z) dx\\
               &\le \begin{cases}  C \int_0^{1/z} |W(x)|^2 dx,&\textup{if } \nu\not =0,\\
                     C  \int_0^{1/z} |W(x)|^2 |\log(xz)| dx,&\textup{if } \nu=0.
\end{cases}
\end{split}
\end{equation}

For $1/z\le x\le 1$ we apply \eqref{eq:Bessel-large1}
and find $|I_\nu(xz)K_\nu(xz)|\le \frac{C}{xz}$, $|I_\nu(xz)|^2\le \frac {C}{xz}e^{2 xz}$. Thus
\begin{equation}\label{eq:ML100128-5}
   \int_{1/z}^1 |W(x)|^2 |I_\nu(xz) K_\nu(xz)|dx \le C \frac 1z \int_{1/z}^1 \frac 1x |W(x)|^2 dx,
\end{equation}
and in view of \eqref{eq:ML100128-3} 
\begin{equation}\label{eq:ML100128-6}
\begin{split}
    |\beta(z)| &\int_{1/z}^1 |W(x)|^2 |I_\nu(xz)|^2 dx\\
        &\le C e^{-2z} \int_{1/z}^1 |W(x)|^2 \frac{1}{xz} e^{2 xz} dx\\
        &\le C\Bigl( e^{-z} \frac 1z \int_{1/z}^{1/2} \frac 1x |W(x)|^2 dx + \frac 1z \int_{1/2}^1 |W(x)|^2 dx\Bigr)\\
       & \le C \frac 1z, \quad z\ge z_0.
\end{split}
\end{equation}
The estimate \eqref{eq:PerturbationEstimate} now follows from \eqref{eq:ML1003051}, \eqref{eq:ML100128-4},
\eqref{eq:ML100128-5}, \eqref{eq:ML100128-6}.

Under the assumptions \eqref{eq:ML100128-2} we apply
Lemma \ref{les-2} to $\frac 1x R(\frac 1x)$ since $\int_{z_0}^\infty \frac 1z |R(z)| dz =\int_0^{1/z_0}
\frac 1x |R(\frac 1x)| dx$ and conclude \eqref{eq:IntCon-a}, \eqref{eq:IntCon-b}.
\end{proof}

We return to the discussion of the operator $H=l_\nu+X\ii V$; recall that $X$
denotes the function $X(x)=x$. We have seen in the previous 
Proposition that $L_\nu$ is a bounded below self--adjoint operator. In fact $L_\nu$ is the Friedrichs extension
of $l_\nu$ restricted to the domain
\begin{equation}\label{eq:ML100128-7}
\dom(l_\nu)=\bigsetdef{f\in\cinfz{(0,1]}}{ B_{1,\theta}f=0}.
\end{equation}
We now want to construct the Friedrichs extension of $H$ on $\dom(l_\nu)$ and compare its
resolvent to that of $L_\nu$; cf. \cite[VI, 2.3]{Kat:PTL}. The problem is that the domains of $L_\nu$ and of the Friedrichs extension
of $H$ on $\dom(l_\nu)$ are not necessarily equal. This is because the domain of $L_\nu$ contains
functions $f(x)$ with $f(x)\sim x^{\nu+1/2}$ as $x\to 0$. For such a function, $X^{-1} V f$ is not necessarily
in $L^2[0,1]$.

\begin{prop}\label{p:mSec}
Let $H=l_\nu+X\ii V$ with $V\in \sV_\nu$ (cf. Def. \plref{def:FunctionSpaces}).
Moreover, let $\dom(l_\nu)$ be given by \eqref{eq:ML100128-7},  $0\le \theta<\pi$, and 
let $q(f,g):=\scalar{l_\nu f}{g}$ be  the form of the operator $l_\nu$. Then the form
\begin{align}\label{38-b}
v(f,g):=\langle X^{-1}Vf, g \rangle_{L^2[0,1]}, \ f,g\in \dom (l_\nu)
\end{align}
is $q$--bounded with arbitrarily small $q$--bound $b$.
\end{prop}
\begin{proof}
We compute for any $g\in \dom (l_\nu)$ and $z\ge z_0$
\[
\begin{split}
|v(g,g)|=& \|x^{-1/2}|V|^{1/2} g\|_{L^2}^2 \\
\le& \|x^{-1/2}|V|^{1/2} (L_\nu+z^2)^{-1/2}\|^2 \scalar{(L_\nu+z^2) g}{g}.
\end{split}
\]
Now Proposition \ref{p:PerturbationEstimate} implies, that for any $b<1$ there exists 
$z\in \R_+$ sufficiently large, such that
\[
|v(g,g)|\leq b \scalar{(L_\nu+z^2)g}{g} = bz^2 \|g\|^2_{L^2}+b q(g,g).\qedhere
\]
\end{proof}

The quadratic form $q$ is bounded below and closable with closure $Q$. 
By the second representation theorem \cite[IV, 2.6 Theorem 2.23]{Kat:PTL}, we have
\begin{align}\label{42-b}
\dom (Q)=\dom ((L_\nu+z_0^2)^{1/2}).
\end{align}

As a consequence of Proposition \ref{p:mSec} 
we find in view of \cite[VI, 1.6, Theorem 1.33]{Kat:PTL} that $(q+v)$ is a sectorial form with
\begin{align}\label{43-b}
\dom (\ovl{q+v})=\dom (Q)=\dom ((L_\nu+z_0^2)^{1/2}).
\end{align}

By the first representation theorem (\cite[VI, 2.1, Theorem 2.1]{Kat:PTL}) it determines uniquely a closed m-sectorial extension 
$H(0,\theta)$ of $H=l_\nu+X^{-1}V$, with domain given by
\begin{equation}
\label{44-b}
\begin{split}
\dom \bigl(H(0,\theta)\bigr)
     &=\bigsetdef{f \in \dom ((L_\nu+z_0^2)^{1/2})}{ (l_\nu+X^{-1}V+z_0^2)f \in L^2[0,1]},\\
     &=\bigsetdef{f\in \dom (\Hmax)}{ c_2(f)=0, B_{1,\theta}f =0}.
\end{split}
\end{equation}
Note that the functional $c_2$ (as well as $c_1$) depends on the potential and the $c_2$ in \eqref{44-b} 
is the one associated to $H$.

\begin{theorem}\label{5-b-t}
The operator $H(0,\theta)$ is m-sectorial, in particular $\spec H(0,\theta)$ is a subset of a sector
$\{\xi\in \C\mid |\arg (\xi-\eta)|\leq\alpha\},$ for some fixed angle $\alpha \in (0, \pi /2)$ and $\eta \in \R$. 
Its resolvent is trace class and 
\begin{align}\label{49-b}
R_1(z):= \|(H(0,\theta)+z)^{-1}-(L_\nu+z)^{-1}\|_{\tr}, \quad z\in \R_+, z>\max(-\eta,0)
\end{align}
satisfies 
\begin{align}
    &\lim_{z\to\infty} z R_1(z)=0,  \quad z\in \R_+,        \label{eq:RemainderEstimateA}\\
    &\int_{z_0}^\infty |R_1(z)|dx <\infty.                    \label{eq:RemainderEstimateB}
\end{align}

Furthermore,
\begin{equation}\label{eq:ML1003054}
\begin{split}
  \Tr\bigl(H(0,\theta)+z\bigr)^{-1}&=\Tr \bigl(L_\nu+z\bigr)^{-1}+R_2(z)\\
                                   &=\frac{a}{\sqrt{z}}+\frac{b}{z}+R_3(z),
\end{split}
\end{equation}
where 
\begin{equation}\label{eq:HeatCoefficientsA}
 a=\frac 12,\quad b=-\frac 12\bigl(\nu+\mu_1(B_{1,\theta})\bigr)
        =\begin{cases} -\frac 12(\nu+\frac 12),&\textup{if } \theta_1=0,\\
         -\frac 12(\nu-\frac 12),&\textup{if } 0<\theta_1<\pi.\\
\end{cases} \textup{ (cf. \eqref{eq:DefMu1})}
\end{equation}
The remainders $R_2(z),R_3(z)$ satisfy \eqref{eq:RemainderEstimateA} and \eqref{eq:RemainderEstimateB} and therefore
the zeta-determinant of $H(0,\theta)$ is well--defined by the formula (see \eqref{eq:ZetaPrimeZero} and Figure \plref{fig:Contour}).
\begin{align}
\log \det\nolimits_{\zeta} H(0,\theta)= -\regint_\Gamma \Tr\bigl((H(0,\theta)+z)\ii\bigr) dz.
\end{align}
\end{theorem}

\begin{proof}
The operator $H(0,\theta)$ is m-sectorial, as it arises from the sectorial 
form $(q+v)$, see \cite[VI.2, Theorem 2.1]{Kat:PTL}. 
Since by Proposition \ref{p:PerturbationEstimate} we have 
\[
\lim_{z\to\infty} \| x^{-1/2} |V|^{1/2} (L_\nu+z)^{-1/2}\|_{\HS}=0,
\]
we may invoke the Neumann series to obtain
\begin{align}\nonumber
&(H(0,\theta)+z)^{-1}-(L_\nu+z^2)^{-1}\\ 
&=\sum\limits_{n\geq 1}(-1)^n (L_\nu+z)^{-\frac{1}{2}}
            \bigl[ (L_\nu+z)^{-\frac{1}{2}}x^{-1}V(L_\nu+z)^{-\frac{1}{2}}\Bigr]^{n} (L_\nu+z)^{-\frac{1}{2}}.\label{52-b}
\end{align}
There is a little subtlety here since $\dom\bigl(H(0,\theta)\bigr)$ does not necessarily equal $\dom(L_\nu)$. However, by 
Proposition \ref{p:mSec} the forms of $H(0,\theta)$ and $L_\nu$ have the same domain. This is used decisively
by writing $(L_\nu+z^2)^{-1/2}$ at the beginning and at the end of \eqref{52-b}.

We estimate the trace norm of the individual summands by
\[
\begin{split}
 \|&(L_\nu+z)^{-1/2}\|^2\cdot \|(L_\nu+z)^{-1/2}x^{-1}V(L_\nu+z)^{-1/2}\|^n_{\tr} \\
   &\leq \|(L_\nu+z)^{-1}\|  \|x^{-1/2}|V|^{1/2}(L_\nu+z)^{-1/2}\|^n_{\HS}\cdot \|(L_\nu+z)^{-1/2}x^{-1/2}|V|^{1/2}\|^n_{\HS}\\
    &\le C z\ii \tilde R(z)^{n}, \label{53a-b}
\end{split}
\]
where $\tilde R(z)= \|x^{-1/2}|V|^{1/2}(L_\nu+z)^{-1/2}\|_{\HS}\cdot \|(L_\nu+z)^{-1/2}x^{-1/2}|V|^{1/2}\|_{\HS}$. 
The claim about $R_1(z)$ now follows from Proposition \plref{p:PerturbationEstimate}.

The first line of \eqref{eq:ML1003054} follows since $| R_2(z) |\le R_1(z)$. As for 
the second line of \eqref{eq:ML1003054} we note that $\Tr(L_\nu+z)\ii$ has a complete asymptotic expansion as $z\to\infty$ \cite{BruSee:RSA}, in particular 
$$\Tr \bigl(L_\nu+z\bigr)^{-1}=\frac{a}{\sqrt{z}}+\frac{b}{z}+O(z^{-3/2}\log z),$$
with $a,b$ as in \eqref{eq:HeatCoefficientsA}. 

For the claim about the zeta-determinant see Section \plref{ss:ZetaDeterminant}. 
\end{proof}

\subsection{General boundary conditions}\label{sec:GenBC}

We now extend Theorem \plref{5-b-t} to general boundary conditions at $0$. Recall
that $0$ is in the limit point case if and only if $\nu\ge 1$. So the following discussion is of relevance
only in the case $\nu<1$. The case $\nu=0$ bears more difficulties (see \cite{FMPS:UPZ}, \cite{KLP:UPR})
and therefore we assume from now on $0<\nu (<1)$.
The difficulty then is that for $0<\theta_0<\pi$ the resolvent of $l_\nu(\theta_0,\theta_1)$ does not absorb 
negative $x$ powers as the operator $l_\nu(0,\theta_1)$ does. 
Therefore, we do not have \eqref{eq:PerturbationEstimate}
at our disposal and hence the resolvent of $H(\theta_0,\theta_1)$ cannot be constructed
as a perturbation of the resolvent of $l_\nu(\theta_0,\theta_1)$. Instead we will employ the 
results about factorizable operators in Section \plref{ss:FactorOperator}. 
However we have to impose a slight restriction on the class of potentials:

\begin{definition}\label{def:DetClass} 
Let $V\in\sV_\nu$ and let $H=l_\nu+X\ii V$ be the corresponding 
regular singular Sturm-Liouville operator. $V$ is called of \emph{determinant class} if
for any pair of admissible boundary conditions $B_{j,\theta_j}$ the 
operator $H(\theta_0,\theta_1)$ satisfies for $z\ge z_0, z\in \R_+$, 
\begin{align}
  \bigl\| \bigl( H(\theta_0,\theta_1)+z\bigr)\ii \bigr\| &= O(|z|^{-1}),
      \label{eq:DetClass1}\\
\bigl \| \bigl( H(\theta_0,\theta_1)+z\bigr)\ii \bigr\|_{\tr} &= O(|z|^{-1/2}),
      \label{eq:DetClass2} \\
\intertext{and for any $\varphi\in \cinfz{(0,1]}$}
\bigl \|\varphi  \bigl( H(\theta_0,\theta_1)+z\bigr)\ii \bigr\|_{L^2\to H^1} &= O(|z|^{-1/2}).
       \label{eq:DetClass3}
\end{align}
Here, $\|\cdot\|_{L^2\to H^1}$ denotes the norm of a map from $L^2[0,1]$ into the first
Sobolev space $H^1[0,1]$.
We denote the set of determinant class potentials by $\sVdet$.
\end{definition}
We note some consequences and give some criteria for $V$ being of determinant class.

\begin{lemma}\label{l:InterpolationInequality}
Let $V\in\sVdet$ and let $W\in L^2_\comp(0,1]$ with $\supp W\subset [\delta,1], \delta>0$.
Then
\begin{equation}\label{eq:InterpolationInequalityA}
    \bigl\| W (H(\theta_0,\theta_1)+z)\ii\bigr\|\le C_{\delta} \|W\|_{L^2}\; |z|^{-2/3}, \quad z\ge z_0.
\end{equation}
For $W\in L^\infty[0,1]$ we have
\begin{equation}\label{eq:InterpolationInequalityB}
    \bigl\| W (H(\theta_0,\theta_1)+z)\ii\bigr\|\le C \|W\|_{\infty} \; |z|^{-1}, \quad z\ge z_0.
\end{equation}
\end{lemma}
\begin{proof} Choose a cut-off function $\varphi\in \cinfz{(0,1]}$ with $\varphi(x)=1$ for $x\ge \delta$.
Then \eqref{eq:DetClass1}, \eqref{eq:DetClass3} and the complex interpolation method \cite[Sec. 4.2]{Tay:PDEI}
yield for $0\le s\le 1$
\begin{equation}\label{eq:ML1003231}
 \bigl \|\varphi  \bigl( H(\theta_0,\theta_1)+z\bigr)\ii \bigr\|_{L^2\to H^s} \le C_s |z|^{-1+s/2}.
\end{equation}
By Sobolev embedding we have $H^s[0,1]\subset C[0,1]$ for $s>1/2$ and thus for these $s$ multiplication by $W$
is continuous $H^s\to L^2$ with norm bounded by $C_s \|W\|_{L^2}$. Combining this with
\eqref{eq:ML1003231} gives
\begin{equation}
  \bigl \|W  \bigl( H(\theta_0,\theta_1)+z\bigr)\ii \bigr\|_{L^2\to L^2} 
       \le  C_{s,\delta} \|W\|_{L^2}  |z|^{-1+s/2}.
\end{equation}
\eqref{eq:InterpolationInequalityA} follows by putting $s=2/3$, \eqref{eq:InterpolationInequalityB}
is obvious from \eqref{eq:DetClass1}.
\end{proof}

\begin{lemma}\label{l:DetClassCriterion} Let $V\in \sVdet$. 
If $W=W_1+W_2, W_1\in L^\infty[0,1], W_2\in L^2_\comp(0,1]$ then $V+X W\in \sVdet$.
\footnote{Note that then $H+W= l_\nu+X^{-1}(V+X W)$.}

Consequently, if $V_1\in \sVdet, V_2\in\sV_\nu$ and $V_1(x)=V_2(x)$
for almost all $x$ in a neighborhood of $0$ then $V_2\in\sVdet$.
Furthermore, there is a constant depending only on $H(\theta_0,\theta_1)$ and the support of $W_2$ such
that for $z\ge z_0$
\begin{equation}\label{eq:ResolventComparison}
\begin{split}
   \bigl\|\bigl(H(\theta_0&,\theta_1)+W+z\bigr )\ii - \bigl(H(\theta_0,\theta_1)+z\bigr)\ii \bigr\|_{\tr}\\
              &  \le C \bigl(\|W_1\|_{\infty}+\|W_2\|_{L^2}\bigr)\; |z|^{-7/6}.
\end{split}
\end{equation}
\end{lemma}
\begin{proof} 
It follows from Lemma \plref{l:InterpolationInequality} that
for $z$ large enough we can employ the Neumann series
\begin{equation}\label{eq:NeumannSeries}
\begin{split}
   (H(\theta_0,&\theta_1)+W+z)\ii - (H(\theta_0,\theta_1)+z)\ii \\
    &=\sum_{n=1}^\infty (-1)^n \bigl(H(\theta_0,\theta_1)+z\bigr)\ii \bigl(W \bigl(H(\theta_0,\theta_1)+z\bigr)\ii\bigr)^n
\end{split}
\end{equation}
and \eqref{eq:DetClass1}, \eqref{eq:DetClass2}, \eqref{eq:DetClass3} follow
for $H(\theta_0,\theta_1)+W$; also \eqref{eq:ResolventComparison} immediately follows.

The second claim follows from the first with $W=X^{-1}(V_2-V_1)\in L^2_\comp(0,1]$.
\end{proof}

\begin{prop}\label{p:DetClassCriterion} Let $V\in \sV_\nu$ be real valued in a neighborhood of $0$. Then $V\in \sVdet$.
\end{prop}
Together with Lemma \ref{l:DetClassCriterion} this shows that at least potentials of the form $V+\gl$, where
$V\in \sV_\nu$ is real valued and $\gl\in \C$, are of determinant class.
\begin{proof} In view of Lemma \plref{l:DetClassCriterion} and Proposition \plref{p:FactorH} we may change
$V$ outside a neighborhood of $0$ such that $V$ becomes real valued everywhere and such that
$H(\theta_0,\theta_1)=D^*D$, where $D$ is a closed extension of $d=-\frac{d}{dx}+\go'/\go$. For the properties of
$\go$ see Proposition \plref{p:FactorH}. Note that since $V$ is real valued we may choose $\go$ to be real valued, too
and hence, in the notation of Proposition \plref{p:FactorH}, $D_1=D_2^*$.

Since $D^*D$ is self--adjoint, elliptic and non-negative \eqref{eq:DetClass1}, \eqref{eq:DetClass3} follow immediately
from the Spectral Theorem. If $\theta_0=0$ then \eqref{eq:DetClass2} follows from Theorem \plref{5-b-t}. If $\theta_0\not=0$
then by Proposition \plref{p:FactorH} the operator $DD^*$ has Dirichlet boundary condition at $0$. Hence
by Theorem \plref{5-b-t} the estimate \eqref{eq:DetClass2} holds for $DD^*$. Since for a non-negative operator
the estimate \eqref{eq:DetClass2} depends only on the spectrum and since $\spec DD^*\cup \{0\}=\spec D^*D\cup\{0\}$
we reach the conclusion.
\end{proof}

Next we prove two comparison results for the asymptotics of the resolvent in
the trace norm. These will then lead to an asymptotic expansion of the trace
of the resolvent for $H(\theta_0,\theta_1)$ for arbitrary admissible boundary
conditions and all determinant class potentials. The technique used in the
first comparison result is well--known for elliptic operators with smooth
coefficients on manifolds (cf. e.g. \cite[Appendix B]{LMP:CCC}). We have to be
slightly more careful here due to the low regularity assumptions on the
potential.

\begin{prop}\label{p:Comparison}
Let $V_j\in \sVdet, j=1,2$ with $V_2-V_1\in L^2_\comp(0,1]$, that is there is
a $\delta>0$ such that $V_1(x)=V_2(x)$ for $0\le x\le \delta$.  Let
$H_j=l_\nu+X\ii V_j$ be the corresponding regular singular Sturm-Liouville
operators and let $B_{\theta_0},B_{\theta_1}$ resp. $B_{\tilde\theta_1}$, be
admissible boundary conditions for $H_j$. Then there is a $z_0\ge 0$ such for
any $\delta'<\delta$ and $z\ge z_0$ the difference
$\bigl(H_1(\theta_0,\theta_1)+z\bigr)^{-1}-\bigl(H_2(\theta_0,\tilde
\theta_1)+z\bigr)^{-1}$
restricted to $L^2[0,\delta']$ is of trace class and 
\[
\bigl\|\bigl((H_1(\theta_0,\theta_1)+z)^{-1}-(H_2(\theta_0,\tilde\theta_1)+z)^{-1}\bigr)\bigr|_{L^2[0,\delta']}\bigr\|_{\tr}=O(|z|^{-3/2}),
\quad z\ge z_0,  z\in \R_+.
\]
\end{prop}

\pagebreak[3]
\begin{proof}
We choose cut-off functions $\phi, \psi\in C^{\infty}_0[0,\delta)$, cf. Figure \plref{fig:CutOff}, such that they are identically one over $[0,\delta']$ and 
\begin{itemize}
\item $\supp(\phi)\subset \supp(\psi)$,
\item $\supp(\phi)\cap \supp(d\psi)=\emptyset$.
\end{itemize}

\begin{figure}[h]
\begin{center}

\begin{tikzpicture}[scale=1.3]
\draw[->] (-0.2,0) -- (7.7,0);
\draw[->] (0,-0.2) -- (0,2.2);

\draw (-0.2,2) node[anchor=east] {$1$};
\draw (0.5,-0.2) node[anchor=north] {$\delta'$} -- (0.5,0.2);
\draw (7.5,-0.2) node[anchor=north] {$\delta$} -- (7.5,0.2);

\draw (0,2) -- (4,2);
\draw (1,2) .. controls (2.4,2) and (1.6,0) .. (3,0);
\draw[dashed] (1,-0.2) -- (1,2.2);
\draw[dashed] (3,-0.2) -- (3,2.2);
\draw (4,2) .. controls (5.4,2) and (4.6,0) .. (6,0);
\draw[dashed] (4,-0.2) -- (4,2.2);
\draw[dashed] (6,-0.2) -- (6,2.2);
\draw (1.5,1) node {$\phi$};
\draw (4.5,1) node {$\psi$};

\end{tikzpicture}

\caption{The cutoff functions $\phi$ and $\psi$.}
\label{fig:CutOff}
\end{center}
\end{figure}
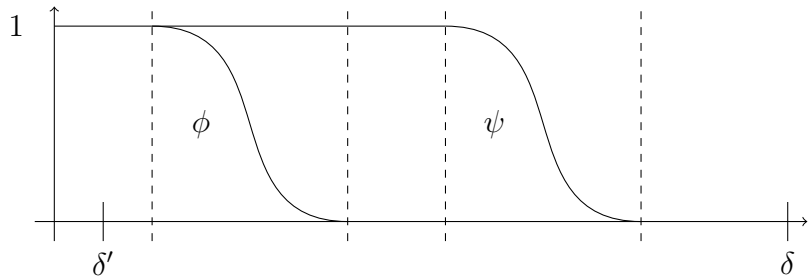

In particular these conditions yield $\psi\phi=\phi$. 
In this proof we will write for brevity $H_1$ instead of $H_1(\theta_0,\theta_1)$
and $H_2$ instead of $H_2(\theta_0,\tilde\theta_1)$.

We now consider 
\[
R(z):=\psi\bigl[(H_1+z)^{-1}-(H_2+z)^{-1}\bigr]\phi.\]
$R(z)$ maps into the domain of $H_j$ and on the support of $\psi$ the
differential expressions $H_1$ and $H_2$ coincide; moreover $\psi \dom(H_1)=\psi\dom(H_2)$. Thus 
$(H_1+z)R(z)=[H_1,\psi]\bigl((H_1+z)^{-1}-(H_2+z)^{-1}\bigr)$.
Arguing similarly for $R(z)^*$ and taking adjoints one then finds
\begin{align*}
(H_1+z)R(z)(H_2+z)=[-\partial_x^2, \psi]\bigl((H_1+z)^{-1}-(H_2+z)^{-1}\bigr)[\partial_x^2,\phi],
\end{align*}
where $[\cdot, \cdot]$ denotes the commutator of the corresponding operators and any function is viewed as a multiplication operator.
Hence
\begin{align*}
R(z)=(H_1+z)^{-1}[-\partial_x^2,\psi]\bigl((H_1+z)^{-1}-(H_2+z)^{-1}\bigr) [\partial_x^2,\phi](H_2+z)^{-1}
\end{align*}
and thus
\begin{align*}
\|R(z)\|_{\tr}\leq \|(H_1&+z)^{-1}\|_{\tr}\Bigl(\|[\partial_x^2,\psi](H_1+z)^{-1}\|+\|[\partial_x^2,\psi](H_2+z)^{-1}\|\Bigr)\cdot\\
                         &\cdot \|[\partial_x^2,\phi](H_2+z)^{-1}\|.
\end{align*} 
By \eqref{eq:DetClass2}
we have $\|(H_1+z)^{-1}\|_{\tr}=O(|z|^{-1/2})$. Let $f$ denote $\psi$ or $\phi$. Then
$[\partial_x^2, f]$ is a first order differential operator whose coefficients are compactly supported in $(0,1)$, hence it maps $H^1[0,1]$
continuously into $L^2_{\comp}(0,1)$. Therefore by \eqref{eq:DetClass3}, with a cut--off function $\chi\in\cinfz{(0,1)}$ with
$\chi=1$ in a neighborhood of $\supp ([\pl_x^2,f])$,
\[
\|[\partial_x^2,f](H_j+z)^{-1}\|\le \|[\pl_x^2,f]\|_{H^1\to L^2}\|\chi (H_j+z)\ii\|_{L^2\to H^1} =O(|z|^{-1/2})
\]
for $z\ge z_0$. Hence $\|[\partial_x^2,\phi](H_j+z)^{-1}\|=O(|z|^{-1/2})$ and $\|[\partial_x^2,\psi](H_j+z)^{-1}\|=O(|z|^{-1/2})$ 
and the proposition is proved.
\end{proof}

We note that in this proof the estimate \eqref{eq:DetClass2} was used only for $H_1$. 

\begin{prop}\label{p:Comparison2} Let $V_j\in\sVdet$, $H_j=l_\nu+X^{-1}V_j,
 j=1,2,$ and let
$B_{0,\tilde\theta_0}, B_{0,\theta_0}, B_{1,\theta_1}$ be admissible boundary conditions.
Then for any $\delta>0$
\[
\bigl\|\bigl((H_1(\theta_0,\theta_1)+z)\ii-(H_2(\tilde\theta_0,\theta_1)+z)^{-1}\bigr)\bigr|_{L^2[\delta,1]}\bigr\|_{\tr}=O(|z|^{-3/2}),
\quad z\ge z_0,  z\in \R_+.
\]
\end{prop}
\begin{proof}
Fix $\delta>0$ and put
\begin{equation}\label{eq:ML100303-2}
  H_3:= -\frac{d^2}{dx^2}+\Bigl(\frac{\nu^2-1/4}{X^2}+\frac{1}{X}V_1\Bigr) \;1_{[\delta/2,1]}
     =: \Delta+q,
\end{equation}
with $q\in L^2_\comp(0,1], \Delta:=-\frac{d^2}{dx^2}$. 

Exactly as in Proposition \ref{p:Comparison} one now shows
\begin{equation}\label{eq:ML100303-3}
\bigl\|\bigl((H_1(\theta_0,\theta_1)+z)\ii-(H_3(\tilde\theta_0,\theta_1)+z)\ii \bigr)\bigr|_{L^2[\delta,1]}\bigr\|_{\tr}=O(|z|^{-3/2}),
\end{equation}
for $z\ge z_0,  z\in \R_+.$ Furthermore, an elementary calculation involving the explicitly computable resolvent kernel of
$\Delta(\tilde \theta_0,\theta_1)$ shows $\| q(\Delta(\tilde\theta_0,\theta_1)+z)\ii\|_{\tr}=O(|z|^{-1/2})$. 
A Neumann series argument then gives
\begin{equation}\label{eq:ML100303-4}
 \bigl\|\bigl((H_3(\tilde\theta_0,\theta_1)+z)\ii-(\Delta(\tilde\theta_0,\theta_1)+z)\ii \bigr)\bigr|_{L^2[\delta,1]}\bigr\|_{\tr}=O(|z|^{-3/2}).
\end{equation}
\eqref{eq:ML100303-3} and \eqref{eq:ML100303-4} imply
for $z\ge z_0,  z\in \R_+$,
\begin{equation}\label{eq:ML100303-5}
 \bigl\|\bigl((H_1(\theta_0,\theta_1)+z)\ii-(\Delta(\tilde\theta_0,\theta_1)+z)\ii \bigr)\bigr|_{L^2[\delta,1]}\bigr\|_{\tr}=O(|z|^{-3/2}).
\end{equation}
The same line of reasoning applies to $H_2$ and hence
\eqref{eq:ML100303-5} also holds with $H_2(\tilde\theta_0,\theta_1)$ instead of $H_1(\theta_0,\theta_1)$, whence the result.
\end{proof}

\begin{theorem}\label{t:TraceAsymptotics}
Let $V\in\sVdet, \nu>0,$ and let $H=l_\nu+X\ii V$ be the corresponding regular singular Sturm-Liouville operator. 
Let $0\le \theta_j<\pi$ ($\theta_0=0$ if $\nu\ge 1$). 

Then the resolvent of  $H(\theta_0,\theta_1)$ is trace class. Moreover,  there
is a $z_0\ge 0$ such that $H(\theta_0,\theta_1) +z$ is invertible for $z\ge z_0$ and
\[\Tr\bigl(H(\theta_0,\theta_1)+z\bigr)^{-1}=\frac{a}{\sqrt{z}}+\frac{b}{z}+R_3(z),\quad z\ge z_0,\]
where  $a= \frac 12,$
\begin{equation}
\begin{split}
     b&=-\frac 12\bigl(\mu(B_{0,\theta_0})+\mu(B_{1,\theta_1})\bigr)=
                      \begin{cases}
                        -\frac 12 \nu-\frac 14, & \textup{if }\theta_0=\theta_1=0,\\
                        -\frac 12 \nu+\frac 14, & \textup{if }\theta_0=0, 0<\theta_1<\pi,\\ 
                         \frac 12 \nu-\frac 14, & \textup{if }0<\theta_0<\pi, \theta_1=0,\\
                         \frac 12 \nu+\frac 14, & \textup{if }0<\theta_0, \theta_1<\pi,
                      \end{cases}
\end{split}
\end{equation}
(cf. \eqref{eq:DefMu1}) is independent of $V$, and the remainder $R_3(z)$ satisfies 
\begin{align}
    &\lim_{z\to\infty} z R_3(z)=0,  \quad z\in \R_+,                                       \label{eq:RemainderEstimateGeneralA}\\
    &\int_{z_0}^\infty |R_3(z)|dx <\infty, \quad z_0>\max
    \spec\bigl(-H(\theta_0,\theta_1)\bigr)\cap \R.   \label{eq:RemainderEstimateGeneralB}
\end{align}
In particular the zeta-determinant of $H(\theta_0,\theta_1)$ is well--defined by the formula 
\eqref{eq:ZetaPrimeZero},
\begin{align}
\log \det\nolimits_{\zeta} L= -\regint_\Gamma \Tr(H(0,\theta)+z)\ii) dz.
\end{align}
\end{theorem}
\begin{proof}
By Proposition \ref{p:FactorH} we may choose a factorizable operator 
$H_1(\theta_0,\theta_1)=D_1D_2$ such that there is a $\delta>0$ 
such that the coefficients of $H_1$ and $H$ coincide on the interval
$[0,\delta]$. Here, $D_1,D_2$ are appropriate closed extensions of
the operators $d_1, d_2$ in \eqref{eq:d1d2} with $\mu=-\nu$.
Then by Propositions \ref{p:Comparison}, \ref{p:Comparison2} we find
\begin{equation}\label{eq:ML100303-6}
   \bigl\|\bigl((H(\theta_0,\theta_1)+z)\ii-(D_1D_2+z)\ii \bigr)\bigr\|_{\tr}=O(|z|^{-3/2}),
\quad z\ge z_0,  z\in \R_+,
\end{equation}
and hence
\begin{equation}
  \Tr\bigl( H(\T_0,\T_1)+z\bigr)\ii =\Tr\bigl( D_1D_2+z\bigr)\ii+O(|z|^{-3/2}).
\end{equation}
We now have to discuss the four possible cases listed in Section \plref{p:Cases}, see also
\eqref{eq:H12}, \eqref{eq:H21}:

\subsubsection*{Case I: $D_1=d_{1,rr}, D_2=d_{2,aa}$} Then $D_1D_2$ has a one-dimensional null space
and $D_2D_1$ is invertible. Applying Theorem \ref{5-b-t} to $D_2D_1$ we obtain 
\begin{equation}
    \Tr\bigl(D_1D_2+z\bigr)\ii = \Tr\bigl(D_2D_1+z\bigr)\ii + z\ii\\
                               = \frac{a}{\sqrt{z}}+\frac{b}{z}+R_3(z),
\end{equation}
where $R_3(z)$ has the claimed properties \eqref{eq:RemainderEstimateGeneralA} and \eqref{eq:RemainderEstimateGeneralB}
and 
$a=1/2, b=-1/2(1-\nu+ 1/2-2)= 1/2(\nu+ 1/2)$.
Note that in formulas involving $D_2D_1$, according to \eqref{eq:H21}, the $\nu$ has to be replaced by $1-\nu$. 

\subsubsection*{Case II: $D_1=d_{1,ra}, D_2=d_{2,ar}$} Then $D_1D_2$ and $D_2D_1$ are both invertible and
hence $\Tr(D_1D_2+z)\ii=\Tr(D_2D_1+z)\ii$, and we can proceed as in Case I.

In the remaining cases III ($D_1=d_{1,ar}, D_2=d_{2,ra}$) and IV ($D_1=d_{1,aa}, D_2=d_{2,rr}$) one can apply
Theorem \ref{5-b-t} directly to $D_1D_2$.
\end{proof}

\section{Variation of the regular singular potential}\label{q-section}
\label{sec:VarPot}

In this section we discuss the behavior of the fundamental system of solutions under 
a certain variation of the potential and derive a variational formula for the zeta-determinant.

\subsection*{Standing assumptions} 
Let $\nu\ge 0, V\in\sVdet$ and let $W_\eta\in L^\infty[0,1]+L^2_\comp(0,1]$
be a family of functions depending on a real or complex parameter $\eta$.
To avoid unnecessary technicalities we assume that $W_\eta$ is of the form
$W_\eta=W_{1,\eta}+W_{2,\eta}$ where  $W_{1,\eta}\in L^\infty[0,1], W_{2,\eta}\in L^2_\comp(0,1]$ satisfy
\begin{enumerate}
\item $\eta\mapsto W_{1,\eta}$ is differentiable as a map into the Banach space $L^\infty[0,1]$,
\item there is a fixed $\delta>0$ such that $\supp W_{2,\eta}\subset [\delta,1]$ and $\eta\mapsto {W_{2,\eta}}\bigl|_{[\delta,1]}$
is differentiable as a map into the Banach space $L^2[\delta,1]$.
\end{enumerate}
For notational convenience we assume $W_0=0$ and put $V_\eta:=V+X W_\eta$ and
\begin{equation}
    H_\eta:=l_\nu +X^{-1} V_\eta = l_\nu+X\ii(V+X W_\eta )=: H_0+W_\eta=: -\frac{d^2}{dx^2}+q_\eta.
\end{equation}
$\eta_0=0$ serves as a base point for a perturbative construction of a fundamental system.

\subsection{Fundamental solutions and their asymptotics}
According to Theorem \plref{5-a-t} let $g_{1,\eta}$ be the unique solution of the 
ODE $H_\eta g_{1,\eta}=0$ with $g_{1,\eta}(x)\sim x^{\nu_1}$, as $x\to 0+$. Note that
the second solution $g_{2,\eta}$ in Theorem \plref{5-a-t} is not uniquely determined
by the requirement $g_{2,\eta}(x)=-\frac{1}{2\nu} x^{\nu_2}$, cf. Remark \plref{rem:ML20100430}. 
Since the solutions now depend on
the parameter $\eta$, the choice of $g_{2,\eta}$ becomes important. Before we specify
$g_{2,\eta}$ we discuss the dependence of $g_{1,\eta}$ on $\eta$. To do so recall
the operator $K_\nu$ from \eqref{eq:VolterraOperatorA} in Section \ref{sec:FunSys}.
For $\ga\ge 0$ consider the Banach space $X^\ga C[0,1]$ with norm
\begin{equation}
      \|f\|_\ga:=\sup_{0\le x\le 1} | x^{-\ga} f(x)|,
\end{equation}
and the Banach space $C^1_\ga[0,1]$ consisting of those functions in $f\in
\bigl(X^\ga C[0,1]\bigr)\cap C^1(0,1]$
with $f'\in X^{\ga-1} C[0,1]$ and norm
\begin{equation}
     \|f\|_{C^1_\ga}:= \|f\|_\ga + \|f'\|_{\ga-1}.
\end{equation}
For $f\in X^\ga C[0,1]$ the inequality \eqref{14-a} gives
\begin{equation}
   |(K_\nu V)^n f(x)| \leq  x^\ga \frac{1}{|\nu|^nn!} \|f\|_{\ga} \Bigl(\int_0^x |V(y)|dy\Bigr)^n, \ \nu\neq 0, \\
\end{equation}
thus $K_\nu V$ is a bounded operator on $X^\ga C[0,1]$ with spectral radius zero. Furthermore, for $f\in X^\ga C[0,1]$
(cf. \eqref{eq:ML1003055}) 
\begin{equation}
    | (K_\nu V f)'(x)|\le x^{\ga -1} \|f\|_\ga \int_0^x |V(y)| dy,
\end{equation}
hence $K_\nu V$ maps $X^\ga C[0,1]$ continuously into $C^1_\ga[0,1]$. 

Recall from \eqref{24-a} that $g_{1,\eta}(x)=x^{\nu_1}(1+\phi_\eta(x))$ with
\begin{equation}
    \phi_\eta=(I-K_\nu V_\eta)\ii K_\nu V_\eta \one.
\end{equation}
Consequently $\phi_\eta$ is differentiable in $\eta$ and
\begin{equation}
\begin{split}
    \pl_\eta \phi_\eta =(I-&K_\nu V_\eta)\ii K_\nu (X \pl_\eta W_\eta)\one \\
                           &+(I-K_\nu V_\eta)\ii K_\nu (X \pl_\eta W_\eta) (I-K_\nu V_\eta)\ii K_\nu V_\eta \one.
\end{split}
\end{equation}
Since $\pl_\eta W_\eta$ is bounded near $0$, 
the operator $K_\nu (X \pl_\eta W_\eta)$ maps $C^1_\ga[0,1]$ continuously into $C^1_{\ga+2}[0,1]$ and hence we have proved

\begin{lemma}\label{l:ML1003061} Under the assumptions stated at the beginning
 of this section,
$g_{1,\eta}$ is differentiable in $\eta$ with
$\pl_\eta g_{1,\eta}(x)=O(x^{\nu_1+2}), \pl_\eta g_{1,\eta}'(x)=O(x^{\nu_1+1})$ as $x\to 0+$.
Moreover, the $O$--constants are locally uniform in $\eta$ and hence
$g_{1,\eta}(x)-g_{1,\eta_0}(x)=O(x^{\nu_1+2})$.
\end{lemma} 
After these preparations we can discuss the second fundamental solution 
$g_{2,\eta}$. For $\eta$ in a neighborhood of $0$ we can fix $x_0\in (0,1)$ such that
$g_{1,\eta}(x)\not=0$ for $0\le x\le x_0$. For these $x$ we note
\begin{equation}\label{18-c}
\begin{split}
g_{1,\eta}(x)^{-2}&-g_{1,0}(x)^{-2}\\
    &=\frac{[g_{1,\eta}(x)+g_{1,0}(x)][g_{1,0}(x)-g_{1,\eta}(x)]}{(g_{1,\eta}(x))^2 (g_{1,0}(x))^2}=O(x^{2-2\nu_1}), \quad x\to 0+\, ,
\end{split}
\end{equation}
where the $O$--constant is independent of $\eta$.
Hence $g_{1,\eta}^{-2}-g_{1,0}^{-2}$ is integrable over $(0,x_0]$ and we put for $x\in (0,x_0)$
\begin{align}
g_{2,0}(x)&= g_{1,0}(x) \int_x^{x_0} g_{1,0}(y)^{-2} dy,\label{21-c}\\
g_{2,\eta}(x)&=-g_{1,\eta}(x)\int_0^x [g_{1,\eta}(y)^{-2}-g_{1,0}(y)^{-2}]dy+\frac{g_{1,\eta}(x)}{g_{1,0}(x)}g_{2,0}(x).\label{eq:ML1003221}
\end{align}
From \eqref{18-c}, \eqref{21-c} and \eqref{eq:ML1003221} we immediately get
\begin{lemma}\label{2-c-t} $g_{1,\eta},g_{2,\eta}$ is a fundamental system of
 solutions for the \ODE\ $H_\eta g=0$
satisfying \eqref{40}, \eqref{40-a}. 
Moreover, $g_{2,\eta}$ is also differentiable in $\eta$ and we have
for $\nu>0$
\begin{align}
g_{2,\eta}(x)&=g_{2,0}(x)+ O(x^{\nu_2+2}), \\
\partial_{\eta}g_{2,\eta}(x)&=O(x^{\nu_2+2}), \quad
\partial_{\eta}g_{2,\eta}'(x)=O(x^{\nu_2+1}),
\end{align}
as $x\to 0$. For $\nu=0$ the estimates are 
$O(x^{5/2-\nu}\log x)=O(x^{5/2}\log x),$  $O(x^{\nu_2+2}\log x)=O(x^{5/2}\log
x),$ $O(x^{\nu_2+1}\log x)=O(x^{3/2}\log x)$, respectively.
\end{lemma}

Lemma \plref{l:ML1003061} and \plref{2-c-t} imply: 

\begin{cor}\label{cor:AsympWronskian} For $\nu>0$
we have the following asymptotics for the Wronskians $W(g_{j,\eta},\pl_\eta g_{k,\eta})=
g_{j,\eta}g_{k,\eta}'-g_{j,\eta}' g_{k,\eta}, j,k=1,2,$ as $x\to 0+$:
\begin{align}
W(g_{1,\eta}, \partial_{\eta} g_{1,\eta})(x)&=O(x^{2\nu_1+1}), \label{28-c} \\
W(g_{2,\eta}, \partial_{\eta} g_{1,\eta})(x)&=O(x^{\nu_1+\nu_2+1})=O(x^2), \label{29-c} \\
W(g_{1,\eta}, \partial_{\eta} g_{2,\eta})(x)&=O(x^{\nu_1+\nu_2+1})=O(x^2), \label{30-c} \\
W(g_{2,\eta}, \partial_{\eta} g_{2,\eta})(x)&=O(x^{2\nu_2+1}). \label{31-c} 
\end{align}
If $\nu=0$ then the estimates are $O(x^2\log x)$ in all four cases.

Hence for $\nu\ge 0$ and all $j,k=1,2,$ we have
\[ 
\lim\limits_{x\to 0} W(g_{j,\eta},\pl_\eta g_{k,\eta})(x)=0.
\]
\end{cor}

Now we are ready to state the variational result which generalizes  
\cite[Prop. 3.4]{Les:DRS} to arbitrary boundary conditions and to more general potentials:
\begin{theorem}\label{thm:PotentialVariation}
\textup{1. } Let $0<\nu<1$, $V\in\sVdet$ and let $\eta\mapsto W_\eta\in L^\infty[0,1]+L^2_\comp(0,1]$ 
be differentiable 
in the sense described at the beginning of this section. Furthermore, let
$0\le \theta_j<\pi, j=0,1$ and
let $H_\eta=l_\nu+X^{-1} V+W_\eta$. Fix $\eta_0$ and let $g_{j,\eta}$ be the fundamental system
constructed above, relative to the base point $\eta_0$ ($g_{j,\eta_0}$ plays the role of the $g_{j,0}$ above).

Then we have $H_\eta(\theta_0,\theta_1)=H_{\eta_0}(\theta_0,\theta_1)+ W_{\eta}-W_{\eta_0}$. 
Moreover, if $H_{\eta_0}(\theta_0,\theta_1)$ is invertible then 
$\eta\mapsto\log\detz H_{\eta}(\theta_0,\theta_1)$
is differentiable at $\eta_0$ and if $\varphi_\eta,\psi_\eta$ denotes a fundamental system which is normalized for
the boundary conditions $B_{j,\theta_j}, j=0,1$ we have
\begin{equation}\label{35-c}
\frac{d}{d\eta}\bigl|_{\eta_0}\log \detz H_\eta(\theta_0, \theta_1)=\frac{d}{d\eta}\bigl|_{\eta_0}\log W(\psi_\eta, \varphi_\eta).
\end{equation} 

\textup{2. } Let $\nu\ge 0$ and let $\eta\mapsto V_\eta\in \sV_\nu$ be differentiable 
(recall from Def. \plref{def:FunctionSpaces} that $\sV_\nu$ is
naturally a Fr\'echet space). Let $H_\eta=l_\nu+X^{-1} V_\eta$, $0\le \theta<\pi$. If $H_{\eta_0}(0,\theta_0)$
is invertible then $\eta\mapsto \log\detz H_{\eta}(0,\theta)$ is differentiable at $\eta_0$ and formula \eqref{35-c}
holds accordingly.
\end{theorem}

\begin{proof}
1. Let $0<\nu<1$. By Lemma \plref{l:ML1003061} and Lemma \plref{2-c-t} we have 
\begin{equation}
\begin{split}
   g_{1,\eta}(x)-g_{1,\eta_0}(x)&=O(x^{5/2}),\\
    g_{2,\eta}(x)-g_{2,\eta_0}(x)&=O(x^{3/2}),
\end{split}
\quad x\to 0.
\end{equation}
Hence by Theorem \plref{thm:LaplacianMaximalGeneral} the domain of $H_{\eta,\max}$ as well as the functionals
$c_1, c_2$ are independent of $\eta$. Thus we have indeed
$H_\eta(\theta_0,\theta_1)=H_{\eta_0}(\theta_0,\theta_1)+ W_{\eta}-W_{\eta_0}$. The proof of Lemma \plref{l:InterpolationInequality}
shows that $W_\eta$ is $H_{\eta_0}(\theta_0,\theta_1)$-bounded and the assumptions on the map $\eta\mapsto W_{\eta}$ then
imply that $\eta\mapsto H_\eta(\theta_0,\theta_1)$ is a graph continuous family of self-adjoint operators; in particular
there is an $\eps>0$ such that $H_\eta(\theta_0,\theta_1)$ is invertible for $|\eta-\eta_0|<\eps$. From now on we 
assume $|\eta-\eta_0|<\eps$. 

From the estimate \eqref{eq:ResolventComparison}  we conclude that 
\begin{equation}
\begin{split}
    \log\detz &H_\eta(\theta_0,\theta_1) -    \log\detz H_{\eta_0}(\theta_0,\theta_1)\\
     &=-\int_\Gamma \Tr\bigl((H_\eta(\theta_0,\theta_1)+z)\ii - (H_{\eta_0}(\theta_0,\theta_1)+z)\ii\bigr) dz
\end{split}
\end{equation}
where the integrand on the right is absolutely summable as it is $O(|z|^{-7/6})$, $z\to \infty$.

Furthermore, according to our assumptions on $W_\eta$ we have
\begin{equation}
\begin{split}
     \frac{d}{d\eta} &\Bigl((H_\eta(\theta_0,\theta_1)+z)\ii - (H_{\eta_0}(\theta_0,\theta_1)+z)\ii \Bigr)\\
        &   =-(H_\eta(\theta_0,\theta_1)+z)\ii (\partial_\eta W_\eta) (H_\eta(\theta_0,\theta_1)+z)\ii.
\end{split}
\end{equation}
By \eqref{eq:ResolventComparison} the trace norm of the right hand side is $O(|z|^{-7/6})$ where the
$O-$constant is locally independent of $\eta$. 
By the Dominated Convergence Theorem we may thus differentiate under the integral and find
\begin{equation}
\begin{split}\label{eq:DetVariation}
\frac{d}{d\eta}&\log \detz  H_\eta(\theta_0, \theta_1)\\
     =&\int_\Gamma \Tr\bigl((H_\eta(\theta_0,\theta_1)+z)\ii (\partial_\eta W_\eta) 
                     (H_\eta(\theta_0,\theta_1)+z)\ii\bigr) dz\\
    = &-\int_\Gamma  \frac{d}{dz} \Tr\bigl( (\partial_{\eta} W_{\eta}) (H_\eta(\theta_0,\theta_1)+z)\ii \bigr)dz\\
    = & \Tr \bigl( (\partial_\eta W_{\eta}\bigr) H_\eta(\theta_0,\theta_1)\ii \bigr).
\end{split}
\end{equation}
Having established this identity we can now proceed as
in the proof of \cite[Prop. 3.4]{Les:DRS}, making essential use of
Corollary \plref{cor:AsympWronskian}.

The kernel $G_\eta(x,y)$ of $H_\eta(\theta_0,\theta_1)$ is given by
\begin{equation}\label{39-c}
G_\eta(x,y)=
W(\psi_\eta, \varphi_\eta)^{-1}\varphi_\eta(x) \psi_\eta(y), \quad x\leq y,
\end{equation}
$W(\psi_\eta, \varphi_\eta)\neq 0$ since $H_\eta(\theta_0,\theta_1)$ is invertible by assumption. 
Since $g_{j,\eta}$ are differentiable in $\eta$ (Lemma \plref{l:ML1003061} and Lemma \plref{2-c-t})
so are $\varphi_\eta, \psi_\eta$. In fact, the normalization condition implies
\begin{equation}
\begin{split}
        \varphi_\eta&=\begin{cases} - \cot \theta_0 \cdot g_{1,\eta}+ g_{2,\eta}, &\textup{if } 0<\theta_0<\pi, \\
                                       g_{1,\eta},  & \textup{if }\theta_0=0.
                      \end{cases}\\
        \psi_\eta   &= a_\eta g_{1,\eta} + b_\eta g_{2,\eta},
\end{split}
\end{equation}
where $a_\eta, b_\eta$ depend differentiably on $\eta$. 
Differentiating the formula $ \varphi''_{\theta, \eta}=q_{\eta}\varphi_{\theta, \eta}$ 
with respect to $\eta$ gives
\begin{equation}
\partial_{\eta} \varphi''_\eta
   =(\partial_{\eta} q_{\eta}) \varphi_\eta +q_{\eta}\partial_{\eta} \varphi_\eta 
   =(\partial_{\eta} W_{\eta}) \varphi_\eta+ q_{\eta}\partial_{\eta} \varphi_\eta, \label{40-c}
\end{equation}
and hence 
\begin{equation}
\begin{split}
(\partial_{\eta}W_{\eta}) \varphi_\eta \psi_\eta
=&(\partial_{\eta} \varphi''_\eta) \psi_\eta - q_{\eta}(\partial_{\eta} \varphi_\eta) \psi_\eta 
=(\partial_{\eta} \varphi_\eta)'' \psi_\eta - (\partial_{\eta} \varphi_\eta) \psi''_\eta \\
=&\frac{d}{dx}\bigl((\partial_{\eta} \varphi_\eta)' \psi_\eta-(\partial_{\eta} \varphi_\eta) 
     \psi'_\eta \bigr) 
=\frac{d}{dx} W(\psi_\eta, \partial_{\eta}\varphi_\eta). \label{50-c}
\end{split}
\end{equation}
Thus we find
\[
\begin{split}
\frac{d}{d\eta}&\log\detz H_\eta(\theta_0,\theta_1)
   =\,\Tr \bigl((\partial_\eta W_\eta ) H_\eta(\theta_0,\theta_1)\ii \bigr) \\
   =&\, W(\psi_\eta, \varphi_\eta)\ii
          \int_0^1 \frac{d}{dx} W(\psi_\eta, \partial_\eta \varphi_\eta)(x) dx \\
   =&\, W(\psi_\eta, \varphi_\eta)\ii
    \bigl( W(\psi_\eta, \partial_\eta \varphi_\eta)(1) - 
    \lim_{x\to 0+}  W(\psi_\eta, \partial_\eta\varphi_\eta)(x)\bigr). 
\end{split}
\]
By Corollary \ref{cor:AsympWronskian} we have 
\begin{equation}\label{60-c}
\lim\limits_{x\to 0+}W(\psi_\eta, \partial_\eta\varphi_\eta)(x)=0.
\end{equation}
On the other hand 
\begin{equation}\label{61-c}
W(\partial_\eta \psi_\eta, \varphi_\eta)(1)=0,
\end{equation}
since $\psi_\eta$ is normalized with $\psi_\eta(1)=0$ and $\psi'_\eta(1)=-1$ 
in case of Dirichlet boundary conditions and 
with $\psi_\eta(1)=1$ in case of generalized Neumann boundary conditions.

Note that in contrast to \cite{Les:DRS}, the proof of relation \eqref{60-c} 
requires a careful asymptotic analysis of the fundamental solutions
as summarized in Corollary \ref{cor:AsympWronskian}.

In view of \eqref{60-c} and \eqref{61-c} we arrive at
\[
\begin{split}
\frac{d}{d\eta}\log \detz H_\eta(\theta_0,\theta_1)&=
  \, W(\psi_\eta, \varphi_\eta)\ii \bigl(
 W(\psi_\eta, \partial_\eta \varphi_\eta)(1)+ W(\partial_\eta \psi_\eta,
 \varphi_\eta)(1)\bigr) \\
  &=\, W(\psi_\eta, \varphi_\eta)^{-1} \frac{d}{d\eta} W(\psi_\eta, \varphi_\eta)
  =\,\frac{d}{d\eta}\log W(\psi_\eta, \varphi_\eta)
\end{split}
\]
and the proof of 1. is complete.

2. For the proof of 2. we only have to note that by Proposition \plref{p:PerturbationEstimate} we can estimate
the trace norm of $(H_\eta(0,\theta)+z)\ii (\partial_\eta V_\eta) (H_\eta(0,\theta)+z)\ii$
by $C |z|\ii R(z) $ where $R(z)$ satisfies \eqref{eq:IntCon-b} and the constant $C$ is locally
independent of $\eta$. Thus we conclude the variation formula \eqref{eq:DetVariation}. The remaining
arguments are then completely analogous to the proof of 1.
\end{proof}

\pagebreak[3]
\begin{remark} One can also prove a variation formula for the dependence of the zeta-determinant
on the boundary conditions $\theta_0,\theta_1$. For the variation of $\theta_1$ at the regular
end this is standard, see e.g. \cite[Prop. 3.6]{Les:DRS}. For the variation of $\theta_0$ the proof
is much more delicate. Due to our approach via factorizable operators the result is not needed and
therefore omitted. However, the factorization method does not extend to matrix valued potentials
in a straightforward way. So, if one would like to generalize the results of this paper to matrix valued
potentials with regular singularities then one would probably need to establish a formula for the
variation of the zeta-determinant under the variation of the boundary conditions at the singular end.
\end{remark}
\subsection{Proof of the Main Theorem \plref{thm:Main}}
We are now finally ready to prove the Main Theorem \plref{thm:Main}. 
As in \cite[Sec. 4]{Les:DRS} we first note that \eqref{eq:Main} is obviously
true if $H(\theta_0,\theta_1)$ is not invertible. Furthermore, if $\varphi(\cdot,z),\psi(\cdot,z)$
denote the normalized solutions for $H(\theta_0,\theta_1)+z$ it follows from Theorem
\plref{thm:PotentialVariation} (surely, for $V\in \sVdet$ the family $z\mapsto V+ z X$ satisfies
the standing assumptions of the beginning of this section)
that $\detz (H(\theta_0,\theta_1)+z)$ and $W(\psi(\cdot,z),\varphi(\cdot,z))$ are holomorphic
functions in $\C$ with the same logarithmic derivative. Hence it suffices to
prove the formula for $H(\theta_0,\theta_1)+z$ for one $z\in\C$.

Let us now first assume that $\theta_0=0,$ i.e. at the left end point we have the Dirichlet boundary condition. 
Except for the low regularity assumptions on the potential this case was treated in
\cite{Les:DRS}. From loc. cit. we will only use the result that 
the formula \eqref{eq:Main} holds for $\theta_0=0$ and $V(x)=x z$, i.e.
for the operator $l_\nu(0,\theta_1)+z$. To reduce the claim
to this case we consider $V_\eta:=\eta V$. 
By Proposition \plref{p:PerturbationEstimate}, $L_\nu:=l_\nu(0,\theta_1)$ is self--adjoint
and bounded below and from \eqref{eq:PerturbationEstimate} we infer that
$H_\eta(0,\theta_1)+z:=L_\nu+\eta X^{-1}V+z$ is invertible for $0\le \eta\le 1$ and
$z\ge z_0$.  Hence we may apply the variation result Theorem \plref{thm:PotentialVariation}, 2.
and we are reduced to the case $V=0$ and thus to \cite{Les:DRS}.

Next we consider the case $0<\theta_0<1$. As noted before this necessarily
means $\nu<1$, since
for $\nu\ge 1$ the left end point is in the limit point case. The case $\nu=0$ is beyond the
scope of this paper and so we assume $0<\nu<1$. By Proposition \plref{p:FactorH} we have
$H(\theta_0,\theta_1)=D_1D_2+W$ with $W\in L^2_\comp(0,1]$ and
$D_1=d_{1,ra}, D_2=d_{2,ar}$ if $\theta_1=0,$ and $D_1=d_{1,rr}, D_2=d_{2,aa}$ 
if $\theta_1>0$. Putting $W_\eta=\eta W, 0\le \eta\le 1$, we infer from
Lemma \plref{l:InterpolationInequality} that $D_1D_2+\eta W +z$ is invertible
for $0\le \eta\le 1$ and $z\ge z_0$, thus invoking again the variation result Theorem \plref{thm:PotentialVariation}, 1. we are reduced to prove the formula \eqref{eq:Main} for the operator $D_1 D_2+z$.
Note that $D_2 D_1$ has the Dirichlet boundary condition at $0$ and hence \eqref{eq:Main}
holds for $D_2 D_1+z$ by the first part of this proof. 
We now look at the cases already discussed in the proof of
Proposition \plref{p:CommutatorWronskian}. We use the notation from loc. cit., in particular
$\varphi_z,\psi_z$ denote a pair of normalized solutions for $(D_2D_1+z)g=0$ and 
$\tilde \varphi_z=\frac{1}{2-2\nu}d_1\varphi_z,\tilde\psi_z=-d_1\psi_z$ the corresponding
pair of normalized solutions for $(D_1D_2+z)g=0$.

Denote by $\mu_0,\mu_1$ the invariants defined in 
\eqref{eq:DefMu0}, \eqref{eq:DefMu1} of the boundary conditions for $D_1D_2$.
Denote by $\mu_j'$ the corresponding invariants for $D_2 D_1$.

\textup{Case I: $0<\theta_1<\pi$.} The kernel of $D_1 D_2$ is one--dimensional and 
$D_2 D_1$ is invertible. We have $\mu_0=-\nu, \mu_1=-1/2$. $\mu_0'=1-\nu, \mu_1'=1/2$. 
Thus, using the proven formula \eqref{eq:Main} for $D_2 D_1$ and Proposition \plref{p:CommutatorWronskian}
\begin{equation}
\begin{split}
    \detz(D_1D_2+z)&=z\detz(D_2D_1+z)\\
     &=\frac{z\pi}{2^{\mu_0'+\mu_1'}\Gamma(\mu_0'+1)\Gamma(\mu_1'+1)} W(\psi_z,\varphi_z)\\
     &= \frac{ 2 (1-\nu)\pi}{2^{\mu_0+\mu_1+2} \Gamma(\mu_0+2)\Gamma(\mu_1+2)}W(\tilde \psi_z,\tilde \varphi_z) \\
     &= \frac{\pi}{2^{\mu_0+\mu_1} \Gamma(\mu_0+1)\Gamma(\mu_1+1)}W(\tilde \psi_z,\tilde \varphi_z).
\end{split}
\end{equation}

\textup{Case II: $\theta_1=0$.} Here $D_1D_2$ and $D_2D_1$ are both invertible and
we have $\mu_0=-\nu,\mu_1=1/2$. $\mu_0'=1-\nu, \mu_1'=-1/2$. 
Thus
\begin{equation}
\begin{split}
    \detz(D_1D_2+z)&=\detz(D_2D_1+z)\\
      &=\frac{\pi}{2^{\mu_0'+\mu_1'}\Gamma(\mu_0'+1)\Gamma(\mu_1'+1)} W(\psi_z,\varphi_z)\\
      &= \frac{2 (1-\nu)\pi}{2^{\mu_0+\mu_1} \Gamma(\mu_0+2)\Gamma(\mu_1)} W(\tilde \psi_z,\tilde \varphi_z) \\
      &= \frac{\pi}{2^{\mu_0+\mu_1} \Gamma(\mu_0+1)\Gamma(\mu_1+1)} W(\tilde \psi_z,\tilde \varphi_z).
\end{split}
\end{equation}
The proof is complete.\hfill\qed

\begin{remark}\label{r:Concluding} We conclude by mentioning that Theorem \plref{thm:Main} can be extended to potentials with
regular singularities at both end points (and otherwise having the same regularity properties as
the class $\sVdet$). The formula \eqref{eq:Main} remains the same. 
For the proof one first employs the factorization method we used here to arrange that, say at the left
end point, one has Dirichlet boundary conditions. For this boundary condition a variation formula for the
variation of the singular potential was proved in \cite[Prop. 3.7]{Les:DRS}. This variation formula is
still valid for our class of potentials and it allows to deform the parameter $\nu$ to $\nu=1/2$. Now one
is basically in the situation with one regular end point and one singular end point and Theorem 
\plref{thm:Main} can be applied.
The details are left to the reader.
\end{remark}

\bibliography{localbib.bib}
\bibliographystyle{amsalpha-lmp}

\end{document}